\documentclass[12pt,oneside]{amsart}
 \usepackage[BCOR=0mm]{typearea}
\usepackage{amsmath,amsfonts,amssymb,amsthm}
 \usepackage{amsaddr}
\usepackage{bbm}
\usepackage{graphicx}
\usepackage{stmaryrd}
\usepackage{enumerate}
\usepackage{mathtools}
\mathtoolsset{showonlyrefs,showmanualtags}
\usepackage[font=footnotesize]{caption}
\usepackage[usenames,dvipsnames,svgnames]{xcolor}
\usepackage{relsize}
\usepackage{stackrel}
\usepackage{hyperref}
\usepackage{mathtools}

\def\Xint#1{\mathchoice
{\XXint\displaystyle\textstyle{#1}}%
{\XXint\textstyle\scriptstyle{#1}}%
{\XXint\scriptstyle\scriptscriptstyle{#1}}%
{\XXint\scriptscriptstyle\scriptscriptstyle{#1}}%
\!\int}
\def\XXint#1#2#3{{\setbox0=\hbox{$#1{#2#3}{\int}$ }
\vcenter{\hbox{$#2#3$ }}\kern-.6\wd0}}
\def\avint{\Xint-}

\newcommand{\mres}{%
  \,\raisebox{-.127ex}{\reflectbox{\rotatebox[origin=br]{-90}{$\lnot$}}}\,%
}
\newcommand{\grad}{\nabla}
\newcommand{\laplace}{\Delta}

\renewcommand{\L}{\mathcal{L}}

\renewcommand{\div}{\grad\cdot}
\newcommand{\N}{\mathbf{N}}
\newcommand{\E}{{\mathcal E}}
\newcommand{\A}{{\mathcal A}}

\newcommand{\R}{\mathbf{R}}

\newcommand{\Z}{\mathbf{Z}}
\newcommand{\g}{ \mathsf{g}}
\renewcommand{\P}{\mathcal{P}}

\newcommand{\M}{{\mathcal M}}

\DeclareMathOperator{\id}{id}

\DeclareMathOperator{\spt}{spt}
\def\loc{{\mathrm{loc}}}

\renewcommand{\H}{\ensuremath{\mathcal{H}}}

\newcommand{\eps}{\varepsilon}

\newtheorem{prop}{Proposition}
\newtheorem{theorem}{Theorem}
\newtheorem{lemma}{Lemma}
\newtheorem{remark}{Remark}
\newtheorem{definition}{Definition}

\newcommand{\tacka}{\, \cdot\,}

\begin{document}

\title{The thin film equation close to self-similarity}
\author{Christian Seis}
\address{Institut f\"ur Angewandte Mathematik, Universit\"at Bonn}
\email{seis@iam.uni-bonn.de}



\date{\today}

\begin{abstract}
In the present work, we study well-posedness and regularity of the multidimensional thin film equation with linear mobility in a neighborhood of the self-similar Smyth--Hill solutions. To be more specific, we perform a von Mises change of dependent and independent variables that transforms the thin film free boundary problem into a parabolic equation on the unit ball. We show that the transformed equation is well-posed and that solutions are smooth and even analytic in time and angular direction. The latter entails the analyticity of level sets of the original equation, and thus, in particular, of the free boundary.
\end{abstract}

\maketitle

\section{Introduction and main results}

\subsection{The background}

In the present work we are concerned with  a thin film equation in arbitrary space dimensions. Our interest is in the simplest case of a linear mobility, that is, we consider the partial differential equation
\begin{equation}
\label{1}
\partial_t u + \div\left(u\grad\laplace u\right)=0
\end{equation}
 in $\R^N$. In this model, $u$ describes the thickness of a viscous thin liquid film on a flat substrate. We will focus on what is usually referred to as  the complete wetting regime, in which the liquid-solid contact angle at the film boundary is presumed to be zero.
Notice that in the three-dimensional physical space, the dimension $N$ of the substrate is $2$.

Equation \eqref{1} belongs to the following family of thin film equations
\begin{equation}
\label{1a}
\partial_t u + \div \left(m(u)\grad \laplace u\right)=0,
\end{equation}
where the mobility factor is given by $m(u) =  u^3 + \beta^{n-3}u^n$ with $\beta $ being the slippage length. The nonlinearity exponent $n>0$ depends on the slip condition at the solid-liquid interface: $n=3$ models no-slip and $n=2$ Navier-slip conditions. The case $n=1$ is a further relaxation and the linear mobility considered here is obtained to leader order in the limit $u\to 0$.

The evolution described in \eqref{1a} was originally derived as a long-wave approximation from the free-surface problem related to the Navier--Stokes equations and suitable model reductions, see, e.g., \cite{OronDavisBankoff97} and the references therein.  At the same time, it can be obtained as the Wasserstein gradient flow of the surface tension energy \cite{Otto98,GiacomelliOtto01,MatthesMcCannSavare09} and serves thus as the natural dissipative model for surface tension driven transport of viscous liquids over solid substrates.

 
The analytical treatment of the equation is challenging and the mathematical understanding is far from being satisfactory. As a fourth order problem, the thin film equation lacks a maximum principle. Moreover, the parabolicity degenerates where $u$ vanishes and, as a consequence, for compactly supported initial data (``droplets''), the solution remains compactly supported \cite{Bernis96,BertschDalPassoGarckeGrun98}. The thin film equation features thus a free boundary given by $\partial\{u>0\}$, which in physical terms is the contact line connecting the phases liquid, solid and vapour.
Nonetheless, by  using  estimates for the surface energy and compactness arguments Bernis and Friedman  established the existence of weak nonnegative solutions over a  quarter of a century ago \cite{BernisFriedman90}. The regularity of these solutions could be slightly improved with the help of certain entropy-type estimates \cite{BerettaBertschDalPasso95,BertozziPugh96,
DalPassoGarckeGrun98}, but this regularity is not sufficient for proving general uniqueness results. To gain an understanding of the thin film equation and its qualitative features, it is thus natural to find and study special solutions first. In the past ten years, most of the attention has been focused on the one-dimensional setting, for instance, near stationary solutions \cite{GiacomelliKnuepferOtto08,GiacomelliKnuepfer10}, travelling waves \cite{GiacomelliGnannKnuepferOtto14,Gnann16}, and self-similar solutions \cite{Gnann15,BelgazemGnannKuehn16}. The only regularity and well-posedness result in higher dimensions available so far is due to John \cite{John15}, whose analyzes the equation around stationary solutions. For completeness, we remark the thin film equation is also studied with non-zero contact angles, e.g., \cite{Otto98,GiacomelliOtto01,GiacomelliOtto03,Knuepfer11,
KnuepferMasmoudi15,Knuepfer15,BelgazemGnannKuehn16,Degtyarev17}. The  latter of these works is particularly interesting as it deals with the multi-dimensional situation.

In the present paper, we will conduct a study similar to John's and investigate the qualitative behavior of solutions close to self-similarity. A family of self-similar solutions to \eqref{1} is given by
 \[
u_*(t,x) = \frac{\alpha_N}{t^{\frac{N}{N+4}}} \left(\sigma_M  -  \frac{|x|^2}{t^{\frac2{N+4}}}\right)_+^2,
 \]
where $\alpha_N =\frac{ 1}{8(N+4)(N+2)}$ and $\sigma_M$ is a positive number that is determined by the mass constraint
\[
\int u_*\, dx = M,
\]
and the subscript plus sign denotes the positive part of a quantity, i.e., $(\cdot)_+ = \max\{0,\cdot\}$. 
These solutions were first found by Smyth and Hill \cite{SmythHill88} in the one-dimensional case and then rediscovered in \cite{FerreiraBernis97}.
As in related parabolic settings, the Smyth--Hill solutions play a distinguished role in the theory of the thin film equation as they are believed to describe the leading order large-time asymptotic behavior of \emph{any} solution---a fact that is currently known only for strong \cite{CarrilloToscani02} and minimizing movement \cite{MatthesMcCannSavare09} solutions. Besides that, these particular solutions are considered to feature the same regularity properties as any ``typical'' solution, at least for large times. Thus, under suitable assumptions on the initial data, we expect the solutions of \eqref{1} to be smooth up to the boundary of their support. (Notice that this behavior is exclusive for the linear mobility thin film equation, cf.~\cite{GiacomelliGnannOtto13}.)

In the present work we consider solutions that are in some suitable sense close to the self-similar Smyth--Hill solution. Instead of working with \eqref{1} directly, we will perform a certain von Mises change of dependent and independent variables, which has the advantage that it freezes the free boundary $\partial\{u>0\}$ to the unit ball. We will mainly address the following four questions:
\begin{enumerate}
\item Is there some uniqueness principle available for the transformed equation?
\item Are solutions smooth?
\item Can we deduce some regularity for the moving interface $\partial\{u>0\}$?
\item Do solutions depend smoothly on the initial data?
\end{enumerate}
We will provide positive answers to all of these questions. In fact, applying a perturbation procedure we will show that the transformed equation is well-posed in a sufficiently small neighborhood of $u_*$. We will furthermore show that the unique solution is smooth in time and space. In fact, our results show that  solutions to the transformed equation are analytic in time and in direction tangential to the free boundary. The latter in particular entails that all level sets and thus also the free boundary line corresponding to the original solutions are analytic. We finally prove analytic dependence on the initial data.

The fact that solutions depend differentiably (or even better) on the initial data will be of great relevance in a companion study on fine large-time asymptotic expansions. Indeed, in \cite{Seis17c}, we investigate the rates at which solutions converge to the self-similarity at any order. Optimal rates of convergence were already found by Carrillo and Toscani \cite{CarrilloToscani02} and Matthes, McCann and Savar\'e \cite{MatthesMcCannSavare09}, and these rates are saturated by spatial translations of the Smyth--Hill solutions. Jointly with McCann \cite{McCannSeis15} we diagonalized the differential operator obtained after formal linearization around the self-similar solution. The goal of \cite{Seis17c} is to translate the spectral information obtained in \cite{McCannSeis15} into large-time asymptotic expansions for the nonlinear problem. For this, it is necessary to rigorously linearize the equation, the framework for which is obtained in the current paper.
This strategy was recently successfully applied to the porous medium equation near the self-similar Barenblatt solutions \cite{Seis14,Seis15b}. The present work parallels in parts \cite{Seis15b} as well as the pioneering work by Koch \cite{Koch99} and the further developments by Kienzler \cite{Kienzler16} and John \cite{John15}

\subsection{Global transformation onto fixed domain}

In this subsection it is our goal to transform the thin film equation \eqref{1} into a partial differential equation that is posed on a fixed domain and that appears to be more suitable for a regularity theory than the original equation. The first step is a customary change of variables that translates the self-similarly spreading Smyth--Hill solution into a stationary solution. This is, for instance, achieved by setting
\[
\hat x = \frac{1}{\sqrt{\sigma_M}} \frac{x}{t^{\frac1{N+4}}},\quad 
\hat t  = \frac{1}{N+4}\log t,\quad
\hat u = \frac{N+4}{\sigma_M^2} t^{\frac{N}{N+4}} u ,
\]
with the effect that the Smyth--Hill solution becomes
\begin{equation}
\label{3}
\gamma \hat u_*(\hat x)^{1/2} = \frac12\left(1-|\hat x|^2\right)_+,\quad \gamma  = \sqrt{2(N+2)},
\end{equation}
and the thin film equation \eqref{1} turns into the \emph{confined} thin film equation
\begin{equation}
\label{2}
\partial_{\hat t}\hat  u + \hat \grad\cdot \left(\hat u\hat \grad \hat \laplace \hat  u\right) - \hat \grad\cdot\left(\hat x\hat u\right)=0.
\end{equation}
It is easily checked that $\hat u_*$ is indeed a stationary solution to \eqref{2} and mass is no longer spreading over all of $\R^N$. Instead, the confinement term pushes all mass towards the stationary $\hat u_*$ at the origin. To simplify the notation in the following, we will drop the hats immediately!

Now that the Smyth--Hill solution became stationary, we will perform a change of dependent and independent variables that parametrizes the solution as a graph over $ u_*$. This type of a change of variables is sometimes referred to as a von Mises transformation. It is convenient to temporarily introduce the variable $v = \gamma u^{1/2}$, so that $\sqrt{2v_*}$ maps the unit ball onto the upper half sphere.
The new variables are obtained by projecting the point $(x,\sqrt{2v(x)})$ orthogonally onto the graph of $\sqrt{2v_*}$: Noting that $\sqrt{2 v(x)+|x|^2}$ is the hypotenuse of the triangle with the edges $(0,0)$, $(0,|x|)$ and $(|x|,\sqrt{2v(x)})$, the projection point has the coordinates $(z,\sqrt{2v_*(z)})$ with
\[
z = \frac{x}{\sqrt{2v(x)+|x|^2}}.
\]
We define the new dependent variable $w$ as the distance of the point $(x,\sqrt{2v(x)})$ from the sphere, that is
\begin{equation}
\label{2a}
1 +  w = \sqrt{2v +|x|^2},
\end{equation}
which entails that $x=(1+w)z$. This change of variables is illustrated in Figure \ref{fig2}.
\begin{figure}[t]
\includegraphics[scale=.8]{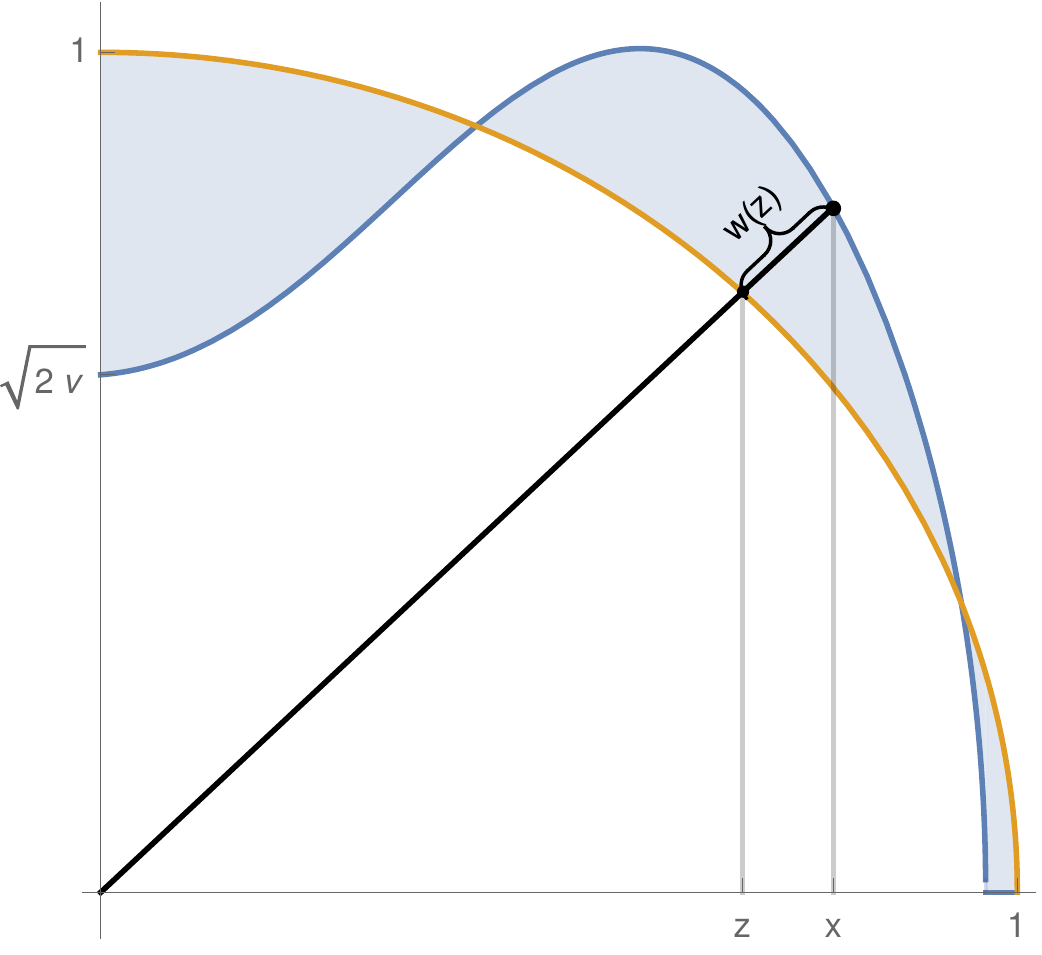}
\caption{\label{fig2} The definition of the $(z,w)$ coordinates.}
\end{figure}

The transformation of the thin film equation \eqref{2} under this change of variables leads to straightforward but tedious computations that we conveniently defer to the appendix. 
The new variable $w$ obeys the equation
\begin{equation}\label{50}
\partial_t w + \L^2w+N\L w  =f[w]
\end{equation}
on $B_1(0)$, where 
\[
\L w = -\rho^{-1}\div\left(\rho^2\grad w\right) = -\rho \laplace w +2z\cdot \grad w
\]
is precisely the linear operator that is obtained by linearizing the porous medium equation $\partial_t u - \laplace u^{3/2}=0$ about the Barenblatt solution, see, e.g., \cite{Seis14,Seis15b}. Before specifying the particular form of the nonlinearity $f[w]$, let us notice that the linear operator $\L^2 +N\L$ corresponding to the thin film dynamics was previously found by McCann and the author \cite{McCannSeis15} by a formal computation. 
Its relation to the porous medium linear operator $\L$ is not surprising but reflects the deep relation between both equations. Indeed, as first exploited by Carrillo and Toscani \cite{CarrilloToscani02}, the dissipation rate of the porous medium entropy is just the surface energy that drives the thin film dynamics. On a more abstract level, this observation can be expressed  by the so-called energy-information relation, first formulated by Matthes, McCann and Savar\'e \cite{MatthesMcCannSavare09}, which connects the Wasserstein gradient flow structures of both equations \cite{Otto98,Otto01,GiacomelliOtto01}.

Let us finally discuss the right-hand side of \eqref{50}. We can split $f[w]$ according to $f[w] = f^1[w] + f^2[w]+f^3[w]$, where
\begin{align*}
f^1[w] &= p\star R_{1}[ w] \star \left((\grad  w)^{2\star} +\grad w\star  \grad^2 w \right)\\
f^2[w]& = p\star R_{1}[ w] \star  \rho\left( (\grad^2 w)^{2\star} + \grad w\star \grad^3 w \right)\\
f^3[w]&= p\star R_{2}[ w] \star  \rho^2\left( (\grad^2 w)^{3\star}+ \grad\tilde w\star  \grad^2\tilde w \star \grad^3  w + (\grad w)^{2\star} \star\grad^4 w\right),
\end{align*}
and
\[
R_i[ w] = \frac{(\grad  w)^{k\star}}{( 1+ w+z\cdot\grad w)^{k+i}}
\]
for some $k\in \N_0$.
We will see in Section \ref{S:Nonlinear} below that the particular form of the nonlinearity is irrelevant for the  perturbation argument. We have  thus introduced a slightly condensed notation to simplify the terms in the nonlinearity: We write $f\star g$ to denote any arbitrary linear combination of the tensors (vectors, matrices) $f$ and $g$. For instance,  $\grad^{m_1} \tilde w \star \grad^{m_2}\tilde  w$ is an arbitrary linear combination of products of derivatives of orders $m_1$ and $m_2$. The iterated application of the $\star$  is abbreviated as $f^{j\star} = f\star \dots \star f$, if the latter product has $j$ factors. The conventions $f^{1\star} = 1\star f$ and $f^{0\star}=1$ apply. We  furthermore use $p$ as an arbitrary representative of a (tensor valued) polynomial in $z$. We have only kept track of those $\rho$ prefactors, that will be of importance later on.

\subsection{The intrinsic geometry and function spaces}

In our analysis of the linear equation associated to \eqref{50}, i.e.,
\begin{equation}
\label{50b}
\partial_t w + \L^2 w+N\L w = f
\end{equation} 
for some general $f$, we will make use of the framework developed earlier in \cite{Koch99,Seis15b} for the second order equation 
\begin{equation}
\label{50a}
\partial_t w + \L w = f.
\end{equation}
The underlying point of view in there is the fact that the previous equation can be interpreted as a heat flow on a weighted manifold, i.e., a Riemannian manifold to which a new volume element (typically a positive multiple of the one induced by the metric tensor) is assigned. The theories for heat flows on weighted manifold parallel those on Riemannian manifolds in many respects, cf.~\cite{Grigoryan06}; for instance, a Calder\'on--Zygmund theory is available for \eqref{50a}. The crucial idea in  \cite{Koch99,Seis15b} is to trade the Euclidean distance on $B_1(0)$ for the geodesic distance induced by the heat flow interpretation. In this way, we equip the unit ball with a non-Euclidean Carnot--Carath\'eodory metric, see, e.g., \cite{BellaicheRisler96}, which has the advantage that the parabolicity of the linear equation can be restored. The same strategy has been applied in similar settings in \cite{DaskalopoulosHamilton98,Koch99,Kienzler16,John15,
DenzlerKochMcCann15,Seis15b}.

We define
\[
d(z,z') =  \frac{|z-z'|}{\sqrt{\rho(z)} + \sqrt{\rho(z')} + \sqrt{|z-z'|}}
\]
for any $z,z' \in B_1(0)$. Notice that $d$ is not a metric as it lacks a proper triangle inequality. This semi-distance is in fact equivalent to the geodesic distance induced on the (weighted) Riemannian manifold associated with the heat flow \eqref{50a}, see \cite[Proposition 4.2]{Seis15b}.
We define open balls with respect to the metric $d$ by
\[
B_r^d(z)  = \left\{ z'\in \overline{B_1(0)}:\: d(z,z')<r\right\},
\]
and set $Q_r^d(z) = \left(\tfrac12r^4,r^4\right)\times B_r^d(z)$ and also $Q(T) = (T,T+1)\times B_1(0)$. Properties of intrinsic balls and volumes will be cited in Section \ref{SS:intrinsicballs} below.

With these preparations, we are in the position to introduce the (semi-)norms
\[
\begin{aligned}
\|w\|_{X(p)} & =  \sum_{(\ell,k,|\beta|)\in \E} \sup_{\substack{ z\in \overline{B_1(0)} \\0<r\le 1}}  \frac{r^{4k+|\beta|-1}}{\theta(r,z)^{2\ell -|\beta|+1}}|Q_r^d(z)|^{-\frac1p} \|\rho^{\ell}\partial_t^k\partial_z^{\beta} w\|_{L^p(Q_r^d(z))}\\
&\quad + \sum_{(\ell,k,|\beta|)\in \E} \sup_{T\ge 1}T\|  \rho^{\ell}\partial_t^k\partial_z^{\beta}  w\|_{L^p(Q(T))},\\
\|f\|_{Y(p)} & =   \sup_{\substack{ z\in \overline{B_1(0)} \\0<r\le 1}}  \frac{r^{3 }}{\theta(r,z)}|Q_r^d(z)|^{-\frac1p} \| f \|_{L^p(Q_r^d(z))} +  \sup_{T\ge 1} T \|    f\|_{L^p(Q(T))},\\
\end{aligned}
\]
for $p\ge 1$, where 
\[
\E = \left\{(0,1,0),(0,0,2),(1,0,3),(2,0,4)\right\}.
\]
These norms (or Whitney measures) induce the function spaces $X(p)$ and $Y(p)$, respectively, in the obvious way.

\subsection{Statement of the results}
	
In view of the particular form of the nonlinearity it is apparent that any well-posedness theory for the transformed equation \eqref{50} requires an appropriate control of $w$ and $\grad w$ to prevent the denominators in $R_i[w]$ from degenerating. This is achieved when the Lipschitz norm $\|w\|_{W^{1,\infty}} = \|w\|_{L^{\infty}} + \|\grad w\|_{L^{\infty}}$ is sufficiently small. 
 A suitable function space for existence and uniqueness is provided by $X(p)\cap L^{\infty}(W^{1,\infty})$ under minimal assumptions on the initial data. Here we have used the convention that $L^q(X) = L^q((0,T);X)$ for some (possibly infinite) $T>0$. Our first main result is:

\begin{theorem}[Existence and uniqueness]
\label{T1}
Let $p>N+4$ be given. There exists $\eps,\eps_0>0$ such that for every $g\in W^{1,\infty}$ with
\[
\|g\|_{W^{1,\infty}}\le \eps_0
\]
there exists a solution $w$ to the nonlinear equation \eqref{50} with initial datum $g$ and $w$ is unique among all solutions with $\|w\|_{L^{\infty}(W^{1,\infty})} + \|w\|_{X(p)}\le \eps$.
Moreover, this solution satisfies the estimate
\[
\|w\|_{L^{\infty}(W^{1,\infty})} + \|w\|_{X(p)} \lesssim \|g\|_{W^{1,\infty}}.
\]
\end{theorem}

Theorem \ref{T1} contains the first (conditional) uniqueness result for the multidimensional thin film equation in a neighborhood of self-similar solutions. Since any solution to the thin film equation is expected to converge towards the self-similar Smyth--Hill solution, our result can be considered as a uniqueness result for large times. 
Notice that the smallness of the Lipschitz norm of $w$ can be translated back into closeness of $u$ to the stationary $u_*$. Indeed, $\|w\|_{W^{1,\infty}}\ll1$ can be equivalently expressed as
\[
\|u-u_*\|_{L^{\infty}(\P(u))}  + \|\grad u + x\|_{L^{\infty}(\P(v))}\ll1,
\]
if $\P(u) = \{u>0\}$ is the positivity set of $u$. Recall that $\grad u_* = -x$ inside $B_1(0)$.

Our second result addresses the regularity of the unique solution found above.

\begin{theorem}\label{T2}
Let $w$ be the solution from Theorem \ref{T1}. Then $w$ is smooth and analytic in time and angular direction.
\end{theorem}

It is clear that the smoothness of $w$ immediately translates into the smoothness of $u$ up to the boundary of its support. Moreover, the analyticity result  particularly implies the analyticity of the level sets of $u$. Indeed, the level set of $u$ at height $\lambda\ge 0$ is given by
\[
\left\{(t,x):\: w(t,r,\phi) = \sqrt{r^2 +\frac{2\sqrt{\lambda}}{\gamma}} - 1\right\},
\]
if $r$ and $\phi$ are the radial and angular coordinates, respectively. As a consequence, the temporal and tangential analyticity of $w$ translates into the analyticity of the level sets of $u$. Notice that the zero-level set is nothing but the free boundary $\partial\{u=0\}$, and thus, Theorem \ref{T2} proves the analyticity of the free boundary of solutions near self-similarity.

In the forthcoming paper \cite{Seis17c}, we will use the gained regularity in time for a construction of invariant manifolds that characterize the large time asymptotic behavior at any order.

The proof of existence and uniqueness in Theorem \ref{T1} follows from a fixed point argument and a maximal regularity theory for the linear equation. Analyticity and regularity is essentially a consequence of an argument first introduced by Angenent \cite{Angenent90a} and later improved by Koch and Lamm \cite{KochLamm12}.

\subsection{Notation}
 One word about constants. In the major part of the subsequent analysis, we will not keep track of constants in inequalities but prefer to use the sloppy notation $a \lesssim b $ if $a\le C b$ for some universal constant $C$. Sometimes, however, we have to include constants like $e^{\pm Ct}$ when dealing with exponential growth or decay rates. In such cases, $C$ will always be a positive constant which is generic in the sense that it will not depend on $t$ or $z$, for instance. This constant might change from line to line, which allows us to write things like $e^{2Ct}\lesssim e^{Ct}$ even for large $t$.

\section{The linear problem}

Our goal in this section is the study of the initial value problem for the linear equation \eqref{50b}. In fact, our analysis also applies to the slightly more general equation \begin{equation}\label{7}
\left\{\begin{array}{rcll}\partial_t w + \L_{\sigma}^2 w + \L_{\sigma} w& = &f\qquad &\mbox{in }(0,\infty)\times B_1(0)\\
w(0,\tacka) &=&g\qquad &\mbox{in }B_1(0)\end{array}\right.
\end{equation}
where $\L_{\sigma}$ corresponds to  the linearized porous medium equation considered in \cite{Seis14,Seis15b}, defined by
\[
\L_{\sigma} w = -\rho^{-\sigma}\div\left(\rho^{\sigma+1} \grad w\right) =  -\rho \laplace w + (\sigma+1) z\cdot \grad w
\]
for any smooth function $w\in C^{\infty}(\bar B)$ and $n\ge 0$ is arbitrary. The constant $\sigma$ is originally chosen greater than $-1$, but we will restrict our attention to the case $\sigma>0$ for convenience. 

Our notion of a weak solution is the following:

\begin{definition}[Weak solution]
Let $0<T\le \infty$ and $f\in L^2((0,T); L^2_{\sigma})$, $g\in L^2_{\sigma}$. We call $w$ a weak solution to \eqref{7} if $w\in L^2((0,T); L^2_{\sigma})$ with $\L w\in L^2((0,T); L^2_{\sigma}))$ solves
\begin{eqnarray}
\lefteqn{-\int_0^T \int \partial_t \zeta w\, d\mu_{\sigma} dt + \int_0^T \int \L_{\sigma}\zeta\L_{\sigma} w\, d\mu_{\sigma} dt + n\int_0^T \int \grad\zeta\cdot\grad w\, d\mu_{\sigma+1} dt}\nonumber \\
& =& \int_0^T \int \zeta f\, d\mu_{\sigma} dt + \int \zeta(0,\tacka)g\, d\mu_{\sigma},\label{8a}
\hspace{13em}
\end{eqnarray}
for all $\zeta \in C^{\infty}([0,\infty)\times \overline{B_1(0)})$ with $\spt\zeta\subset [0,T)\times \overline{B_1(0)}$.
\end{definition}

Here we have used the notation $L^2_{\sigma}$ for  Lebesgue space $L^2(\mu_{\sigma})$, if $\mu_{\sigma}$ is the absolute continuous measure defined by
\[
d \mu_{\sigma} = \rho^{\sigma}\, dx.
\]

The Hilbert space theory for \eqref{7} is relatively easy and will be developed in Subsection \ref{SS:energy} below. In order to perform a perturbation argument on the nonlinear equation, however, we need to control the solution $w$ in the Lipschitz norm. The function spaces $X(p)$ and $Y(p)$ introduced earlier are suitable for such an argument. In fact, our objective in this section is the following result for the linear equation \eqref{7}.

\begin{theorem}
\label{Thm1}
Let $p>N+4$ be given. Assume that $g\in W^{1,\infty}$. Then there exists a unique weak solution $w$ to \eqref{7}. and this solution satisfies the a priori estimate
\[
\|w\|_{W^{1,\infty}} + \|w\|_{ X(p)} \lesssim \|f\|_{Y(p)} + \|g\|_{W^{1,\infty}}.
\]
\end{theorem}

As mentioned earlier, a change from the Euclidean distance to a Carnot--Carath\'e\-o\-do\-ry distance suitable for the second order operator $\L_{\sigma}$ will be crucial for our subsequent analysis.  In the following subsection we will recall some basic properties of the corresponding intrinsic volumes and balls, which were derived earlier in \cite{Seis15b}. Section \ref{SS:prelim} intends to provide some tools that allow to switch from the spherical setting to the Cartesian one. Energy estimates are established in Subsection \ref{SS:energy}. In Subsection \ref{SS:hom} we treat the homogeneous problem and derive Gaussian estimates. A bit of Calder\'on--Zygmund theory is provided in Subsection \ref{SS:CZ}. Finally, Subsection \ref{SS:inhom} contains the theory for the inhomogeneous equation.

\subsection{Intrinsic balls and volumes}\label{SS:intrinsicballs}

 In the following, we will collect some definitions and properties that are related to our choice of geometry and that will become relevant in the subsequent analysis of the linearized equation. Details and derivations can be found in \cite[Chapter 4]{Seis15b}.

It can be shown that the intrinsic balls $B_r^d(z)$ are equivalent to Euclidean balls in the sense that
there exists a constant $C<\infty$ such that
\begin{equation}
\label{8d}
B_{C^{-1} r\theta(r,z)}(z) \subset B_r^d(z) \subset B_{Cr\theta(r,z)}(z)
\end{equation}
for any $z\in  \overline {B_1(0)}$. Here $\theta$ is defined by
\[
\theta (r,z) = r \vee \sqrt{\rho(z)}.
\]
For the local estimates, it will be crucial to notice that
\begin{equation}
\label{8b}
\sqrt{\rho(z_0)}\lesssim r\quad\Longrightarrow\quad \sqrt{\rho(z)} \lesssim r \quad \mbox{for all }z\in B_r^d(z_0),
\end{equation}
and
\begin{equation}
\label{8c}
\sqrt{\rho(z_0)}\gg r\quad\Longrightarrow\quad \rho(z) \sim \rho(z_0) \quad \mbox{for all }z\in B_r^d(z_0).
\end{equation}
In particular, it holds that $\theta(r,\cdot)\sim \theta(r,z_0)$ in $B_r^d(z_0)$. Moreover, if $\sqrt{\rho(z_0)}\lesssim r$, and $z_0\not =0$, then \eqref{8d} implies that
\begin{equation}
\label{8e}
 B_{C^{-1}r^2}\left(\frac{z_0}{|z_0|}\right) \cap \overline{ B_1(0)} \subset B_r^d(z_0) \subset B_{Cr^2} \left(\frac{z_0}{|z_0|}\right) \cap \overline{B_1(0)} .
\end{equation}

We will sometimes write $|A|_{\sigma} = \mu_{\sigma}(A)$ for measurable sets $A$.
The volume of an intrinsic ball can be calculated as follows: 
\begin{equation}
\label{8f}
|B_r^d(z)|_{\sigma} \sim r^N\theta(r,z)^{N+2\sigma}.
\end{equation}
In particular, it holds that
\begin{equation}\label{8g}
\frac{|B_r^d(z)|_{\sigma}}{|B_r^d(z')|_{\sigma}} \lesssim \left(1+\frac{d(z,z')}r\right)^{2N+2\sigma}.
\end{equation}

\subsection{Preliminary results}\label{SS:prelim}

By the symmetry of $\L_{\sigma}$ and the Cauchy--Schwarz inequality in $L^2_{\sigma}$, we have the interpolation 
\begin{equation}
\label{23}
\|\grad \zeta\|_{\sigma+1}^2  = \int\zeta\L_{\sigma}\zeta\, d\mu_{\sigma}\le \|\zeta\|_{\sigma} \|\L_{\sigma}\zeta\|_{\sigma}.
\end{equation}

For further reference, we also quote the maximal regularity estimate
\begin{equation}
\label{9}
\|\grad w\|_{\sigma} + \|\grad^2 w\|_{\sigma+2} \lesssim \|\L_{\sigma} w\|_{\sigma},
\end{equation}
cf.\ \cite[Lemma 4.6]{Seis15b}.

Close to the boundary, the operator $\L_{\sigma}$ can be approximated with the linear operator studied in \cite{Kienzler16},
\[
\tilde \L_{\sigma} w = -z_N^{-\sigma}\div\left(z_N^{\sigma+1}\grad w\right) = -z_N\laplace w -(\sigma+1)\partial_N w.
\]
This operator is considered on the halfspace $\R^{N}_+ = \{z\in \R^N:\: z>0\}$. Defining $\tilde \mu_{\sigma} = z_N^{\sigma}dz$ for any $\sigma>-1$ and using the notation $\|\tacka\|_{\sigma}$ for the norm on $L^2(\tilde \mu_{\sigma})$ (slightly abusing notation), then we have, analogously to \eqref{9} that
\begin{equation}
\label{10}
\|\grad w\|_{\sigma} + \|\grad^2 w\|_{\sigma+2} \lesssim \|\tilde \L_{\sigma} w\|_{\sigma}.
\end{equation}

Our first lemma shows how the second order elliptic equation can be transformed onto a problem on the half space.

\begin{lemma}\label{Lemma2}
Suppose that
\[
\L_{\sigma} w = \xi
\]
for some $w$ such that $\spt(w)\subset \overline{B_1(0)}\cap B_{\eps}(e_N)$ for some $\eps>0$. Let $\Phi(z) = \sqrt{1-|z'|^2} e_N - z$ for $z\in \overline{B_1(0)}$. If $\eps$ is sufficiently small, then $\Phi$ is a diffeomorphism on $ \overline{B_1(0)}\cap B_{\eps}(e_N)$. Moreover, $\tilde w$ defined by $\tilde w(\Phi(z)) = w(z)$ solves the equation
\[
\tilde \L_{\sigma} \tilde w = \tilde \xi,
\]
where $A=A(\tilde z)\in \R^{N\times N}$ and $b= b(\tilde z)\in\R^N$ are smooth functions with
\[
\tilde \xi = \xi \circ\Phi^{-1} + A: \tilde\grad^2 \tilde w + b\cdot\tilde \grad\tilde w
\]
with $|A(\tilde z)|\lesssim |\tilde z|^2,\, |b(\tilde z)| \lesssim |\tilde z|$.
\end{lemma}

\begin{proof}
It is clear that $\Phi$ is a diffeomorphism from a small ball around $z=e_N$ into a small neighborhood around the origin in $\R^N_+$. Moreover, a direct calculation and Taylor expansion show that
\begin{eqnarray*}
  \laplace  w &=& \tilde \laplace \tilde w + \frac{2\tilde z'}{\sqrt{1-|\tilde z'|^2} } \cdot\tilde \grad'\tilde\partial_N \tilde w  + \frac{|\tilde z'|^2}{1-|\tilde z'|^2} \tilde \partial_N^2 \tilde w\\
  &&\mbox{} -\left(\frac{N-1}{\sqrt{1-|\tilde z'|^2}} + \frac{|\tilde z'|^2}{(1-|\tilde z'|^2)^{\frac32}}\right)\tilde \partial_N \tilde w\\
  &=& \tilde \laplace \tilde w + A_1(\tilde z): \tilde \grad^2 \tilde w + b_1(\tilde z)\cdot \tilde \grad \tilde w ,\\
  z\cdot\grad w &=& -\sqrt{1-|\tilde z'|^2}\tilde \partial_N \tilde w  + \tilde z'\cdot \tilde \grad'\tilde w + \left(\tilde z_N -\frac{|\tilde z'|^2}{\sqrt{1-|\tilde z'|^2}}\right)\tilde \partial_N\tilde w\\
  &=& -\tilde \partial_N \tilde w + b_2(\tilde z)\cdot \tilde \grad\tilde w,\\
 \rho(z) &=& \sqrt{1-|\tilde z'|^2}\tilde z_N  - \frac12\tilde z_N^2\ = \ \tilde z_N + c(\tilde z),
\end{eqnarray*}
where $|A_1(\tilde z)|\lesssim |\tilde z|$, $|b_1(\tilde z)|\lesssim 1$, $|b_2(\tilde z)|\lesssim |\tilde z|$, $|c(\tilde z)|\lesssim |\tilde z|^2$. We easily infer the statement.
\end{proof}

A helpful tool in the derivation of the $L^2_{\sigma}$ maximal regularity estimates for our parabolic problem \eqref{7} will be the following estimate for the Cartesian problem.

\begin{lemma}\label{Lemma3}
Suppose $\tilde w$ is a smooth solution of the equation
\[
\tilde \L_{\sigma} \tilde w = \tilde \xi
\]
for some smooth $\xi$. Then
\[
\|\tilde \grad^2 \tilde w\|_{\sigma} + \|\tilde\grad^3\tilde w\|_{\sigma+2} + \|\tilde \grad^4\tilde w\|_{\sigma+4} \lesssim \|\tilde \grad \tilde \xi\|_{\sigma} + \|\tilde \grad^2\tilde \xi\|_{\sigma+2}.
\]
\end{lemma}

\begin{proof}
We start with the derivation of higher order tangential regularity. Since $\tilde \L_{\sigma}$ commutes with $\tilde \partial_i$ for any $i\in\{1,\dots,N-1\}$, differentiation yields $\tilde \L_{\sigma} \tilde \partial_i \tilde w = \tilde \partial_i\tilde \xi$, and thus via \eqref{10},
\[
\|\tilde \grad\tilde \partial_i \tilde w\|_{\sigma} + \|\tilde \grad^2 \tilde \partial_i \tilde w\|_{\sigma+2} \lesssim \|\tilde \grad\tilde \xi\|_{\sigma}.
\]
We take second order derivatives in tangential direction and rewrite the resulting equation as $\tilde \L_{\sigma+2} \tilde \partial_{ij} \tilde w = \tilde \partial_{ij} \tilde \xi - 2\tilde\partial_{ijN}\tilde w$, where $i,j\in\{1,\dots,N-1\}$. From the above estimate and \eqref{10} we obtain
\[
\|\tilde \grad \tilde\partial_{ij} \tilde w\|_{\sigma+2} + \|\tilde \grad^2\tilde \partial_{ij} \tilde w\|_{\sigma+4} \lesssim \|\tilde \grad\tilde \xi\|_{\sigma} + \|\tilde \grad^2\tilde \xi\|_{\sigma+2}.
\]
Transversal derivatives do not commute with $\tilde \L_{\sigma}$. Instead, it holds that $\tilde \partial_N \tilde \L_{\sigma} = \tilde \L_{\sigma+1} \tilde \partial_N  - \tilde \laplace'$. A double differentiation in transversal direction yields thus $\tilde \L_{\sigma+2}\tilde \partial_N^2 \tilde w = \tilde \partial_N^2 \tilde \xi  + 2\tilde \laplace' \tilde \partial_N\tilde w$. We invoke \eqref{10}  and obtain with the help of the previous estimate
\[
\|\tilde \grad\tilde \partial_N^2 \tilde w\|_{\sigma+2} + \|\tilde\grad^2 \tilde \partial_N^2 \tilde w\|_{\sigma+4} \lesssim  \|\tilde \grad \tilde \xi\|_{\sigma} + \|\tilde \grad^2\tilde \xi\|_{\sigma+2}.
\]
Finally, the control of $\tilde \partial_N^2 \tilde w$ follows by using the transversally differentiated equation in the sense of $(\sigma+2) \tilde \partial_N^2 \tilde w = -\tilde \partial_N\tilde \xi - \tilde z_N\tilde \laplace\tilde \partial_N\tilde w - \tilde \laplace'\tilde w$ and the previous bounds.
\end{proof}

\subsection{Energy estimates}\label{SS:energy}

In this subsection, we derive the basic well-posedness result, maximal regularity estimates and local estimates in the Hilbert space setting. We start with existence and uniqueness.

\begin{lemma}\label{Lemma1}
Let $0< T\le \infty$ and $f\in L^2((0,T); L^2_{\sigma})$, $g\in L^2_{\sigma}$. Then there exists a unique weak solution to \eqref{7}. Moreover, it holds
\[
\sup_{ (0,T)}\|w\|_{\sigma}^2 + \int_0^T \|\grad w\|_{\sigma}^2\, dt + \int_0^T \|\grad^2 w \|_{\sigma+2}^2\, dt\lesssim \int_0^T \|f\|_{\sigma}^2\, dt + \|g\|_{\sigma}^2.
\]
\end{lemma}

\begin{proof}
Existence of weak solutions can be proved, for instance, by using an implicit Euler scheme. Indeed, thanks to \eqref{9}, it is easily seen that for any $h>0$ and $f\in L^2_{\sigma}$ the elliptic problem 
\[
\frac1h w +(\L_{\sigma}^2 +n\L_{\sigma})  w =f
\]
has a unique solution $w$ satisfying $\L_{\sigma}w \in L^2_{\sigma}$, see \cite[Appendix]{Seis14} for the analogous second order problem. This solution satisfies the a priori estimate
\[
\frac1{\sqrt{h}} \|w\|_{\sigma} + \|\L_{\sigma} w\|_{\sigma}  \lesssim \|f\|_{\sigma}.
\]
With these insights, it is an exercise to construct time-discrete solutions to \eqref{7}, and standard compactness arguments allow for passing to the limit, both in the equation and in the estimate.  In view of the linearity of the equation, uniqueness follows immediately. This concludes the proof of the lemma.
\end{proof}

Our next result is a maximal regularity estimate for the homogeneous problem.

\begin{lemma}\label{Lemma4}
Let $w$ be a solution to the initial value problem \eqref{7} with $g=0$ and $f\in L^2((0,T);L^2_{\sigma}))$ for some $0<T\le\infty$. Then the mappings $t\mapsto \|w(t)\|_{\sigma}$ and $t\mapsto \|\grad w(t)\|_{\sigma+1}$ are continuous on $[0,T]$ and $\grad^2w, \grad^3 w, \rho\grad^4 w\in L^2((0,T);L_{\sigma}^2))$ with
\[
\|\partial_t w\|_{L^2(L^2_{\sigma})} + \|\grad^2 w\|_{L^2(L^2_{\sigma})} + \|\grad^3 w\|_{L^2(L^2_{\sigma+2})} + \|\grad^4 w\|_{L^2(L^2_{\sigma+4})} \lesssim \|f\|_{L^2(L^2_{\sigma})}.
\]
\end{lemma}
In the statement of the lemma, we have written $L^2(L^2_{\sigma})$ for $L^2((0,T);L^2_{\sigma})$.
\begin{proof}
We perform a quite formal argument that can be made rigorous by using the customary approximation procedures. Choosing $\zeta = \chi_{[t_1,t_2]} \partial_t w$ as a test function in \eqref{8a}, we have the identity
\begin{eqnarray*}
\lefteqn{\int_{t_1}^{t_2} \|\partial_t w\|_{\sigma}^2\, dt +\frac12 \|\L_{\sigma} w(t_2)\|_{\sigma}^2 +\frac{n}2 \|\grad w(t_2)\|_{\sigma+1}}\\
&=& \int_{t_1}^{t_2} \int f \partial_t w\, d\mu_{\sigma}dt + \frac12 \|\L_{\sigma} w(t_1)\|_{\sigma} + \frac{n}2 \|\grad w(t_1)\|_{\sigma+1}.
\end{eqnarray*}
Combining this bound with the estimate from Lemma \ref{Lemma1}, we deduce that the mappings $t\mapsto \|w(t)\|_{\sigma}$ and $t\mapsto \|\grad w(t)\|_{\sigma+1}$ are continuous. Moreover,
\[
\int_0^T \|\partial_t w\|_{\sigma}^2\, dt \lesssim \int_0^T \|f\|^2_{\sigma}\, dt
\]
because of $g=0$. A similar estimate holds for $\L_{\sigma}w$ by the virtue of Lemma \ref{Lemma1}, and the statement thus follows upon proving
\begin{equation}
\label{11}
\|\grad^2 w\|_{\sigma}  + \|\grad^3 w\|_{\sigma+2} + \|\grad^4 w\|_{\sigma+4}\lesssim \|\xi\|_{\sigma}+ \|\grad \xi\|_{\sigma} + \|\grad^2 \xi\|_{\sigma+2} ,
\end{equation}
for any solution of the elliptic problem $\L_{\sigma} w = \xi$, because the right-hand side is bounded by $\|f\|_{\sigma}$ thanks to \eqref{9} and Lemma \ref{Lemma1}. It is not difficult to obtain estimates in the interior of $B_1(0)$. For instance, since  $\L_{\sigma+2} \partial_i w = \partial_i \xi - z_i \laplace w + 2z\cdot \grad \partial_i w -(\sigma+1) \partial_i w $ is bounded in $L^2_{\sigma+2}$ and because $L^2_{\sigma}\subset L^2_{\sigma+2}$,  an application of \eqref{9} yields that 
\begin{equation}
\label{12}
\|\grad^2 w\|_{\sigma+2} + \|\grad^3 w\|_{\sigma+4}\lesssim \|\partial_i \xi \|_{\sigma}  + \|\grad w\|_{\sigma} + \|\grad^2 w\|_{\sigma+2}\lesssim \|\xi\|_{\sigma} + \|\grad \xi\|_{\sigma}.
\end{equation}
Since $\rho \sim 1$ in the interior of $B_1(0)$, this estimate entails the desired control of the second and third order derivatives in the interior of $B_1(0)$. Fourth order derivatives can be estimated similarly.

To derive estimates at the boundary of $B_1(0)$, it is convenient to locally flatten the boundary. For this purpose, we localize the equation with the help of a smooth cut-off function $\eta$ that is supported in a small ball centered at a given boundary point, say $e_N$,
\[
\L_{\sigma} (\eta w) = \eta \xi -2\rho \grad \eta\cdot \grad w -\rho \laplace \eta w +(\sigma+1)z\cdot \grad \eta w=:\tilde \xi.
\]
A short computation shows that $\|\tilde \xi\|_{\sigma} + \|\grad \tilde \xi \|_{\sigma} + \|\grad^2\tilde \xi\|_{\sigma+2}\lesssim \| \xi\|_{\sigma} + \|\grad  \xi \|_{\sigma} + \|\grad^2\xi\|_{\sigma+2} $, where we have used \eqref{9} and \eqref{12} and the Hardy--Poincar\'e inequality $\|w\|_{\sigma}\lesssim \|\grad w\|_{\sigma+1}$ from \cite[Lemma 3]{Seis14}. (For this, notice that we can assume that $w$ has zero average because solutions to $\L_{\sigma} w=\xi$ are unique up to constants.) Establishing \eqref{11} for this localized equation is now a straight forward calculation based on the transformation from Lemma \ref{Lemma2} and the a priori estimate from Lemma \ref{Lemma3}. A covering argument concludes the proof.
\end{proof}

A crucial step in the derivation of the Gaussian estimates is the following local estimate.

\begin{lemma}\label{Lemma5}
Let $0<\hat \eps < \eps<1$ and $0<\delta <\hat \delta<1$ be given. Let $w$ be a solution to the inhomogeneous equation \eqref{7}. Then the following holds for any $z_0\in \overline{B_1(0)}$, $\tau\ge 0$ and $0<r\lesssim 1$:
\begin{eqnarray*}
\lefteqn{\iint_{Q} (\partial_t w)^2\, d\mu_{\sigma}dt +  \frac{\theta(r,z_0)^4}{r^4}\iint_
Q |\grad^2 w|^2\, d\mu_{\sigma}dt}\\
&&\mbox{} + \frac{\theta(r,z_0)^2}{r^2} \iint_Q |\grad^3 w|^2\, d\mu_{\sigma+2}dt + \iint_Q |\grad^4 w|^2\, d\mu_{\sigma+4} dt\\
&&\mbox{}\hspace{1em}\lesssim \iint_{\widehat Q} f^2\, d\mu_{\sigma}dt + \frac1{r^8} \iint_{\widehat Q} w^2 + r^2\theta(r,z_0)^2 |\grad w |^2\, d\mu_{\sigma}dt,
\end{eqnarray*}
where $Q = (\tau + \eps r^2,\tau +r^2)\times B_{\delta r}^d(z_0)$ and $\widehat Q = (\tau + \hat \eps r^2,\tau +r^2)\times B_{\hat \delta r}^d(z_0)$.

\end{lemma}

\begin{proof} Because \eqref{7} is invariant under time shifts, we may set $\tau=0$.
We start recalling that
\[
\L_{\sigma}(\eta w) = \eta\L_{\sigma} w-2\rho\grad \eta\cdot\grad w + (\L_{\sigma} \eta)w,
\]
for any two functions $\eta$ and $w$, and thus, via iteration,
\[
\L_{\sigma}^2(\eta w) = \eta\L^2_{\sigma} w - 2\rho\grad\eta\cdot\grad\L_{\sigma} w + \L_{\sigma} \eta \L_{\sigma} w -2\L_{\sigma}\left(\rho \grad\eta\cdot\grad w\right) + \L_{\sigma}\left((\L_{\sigma} \eta) w\right).
\]
In the sequel, we will choose $\eta$ as a smooth cut-off function that is supported in the intrinsic space-time cylinder $\widehat Q$, and constantly $1$ in the smaller cylinder $ Q$. For such cut-off functions, it holds that
$|\partial_t^k\partial_z^{\beta} \eta|\lesssim r^{-2k-|\beta|}\theta(r,z_0)^{-|\beta|}$. (Here and in the following, the dependency on $\eps,\hat\eps,\delta$ and $\hat \delta$ is neglected in the inequalities.) Then $\eta w$ solves the equation
\begin{eqnarray*}
\lefteqn{\partial_t(\eta w) +\L^2_{\sigma}(\eta w) + n\L_{\sigma} (\eta w) }\\
&=& \eta f + \partial_t \eta w -  4 \L_{\sigma}\left(\rho \grad\eta\cdot\grad w\right) -  4\rho \grad(\rho\grad \eta) : \grad^2 w \\
&&\mbox{}-2\rho z\cdot \grad \eta \laplace w - 2(\sigma+1+n) \rho\grad \eta\cdot \grad w + 2\L_{\sigma}(\rho\grad \eta)\cdot \grad w \\
&&\mbox{}  +\L_{\sigma}((\L_{\sigma} \eta)w)+ \L_{\sigma} \eta \L_{\sigma} w +n (\L_{\sigma} \eta) w.
\end{eqnarray*}
For abbreviation, we denote the right-hand side by $\tilde f$. Testing against $\eta w$ and using the symmetry and nonnegativity properties of $\L_{\sigma}$ and the fact that $\eta w=0$ initially, we obtain the estimate
\[
\iint (\L_{\sigma} (\eta w))^2\, d\mu_{\sigma} dt  \le \iint \eta w\tilde f\, d\mu_{\sigma}dt.
\]
A tedious but straightforward computation then yields
\begin{eqnarray}
\left|\int \eta w\tilde f\, d\mu_{\sigma}\right| &\lesssim & \left( r^2 \|\chi f\|_{\sigma} + \|\L_{\sigma}(\eta w)\|_{\sigma} +\frac1{r^2} \|\chi w\|_{\sigma} + \frac{\theta_0}{r} \|\chi \grad w\|_{\sigma} \right)\nonumber\\
&&\mbox{} \quad \times \left(\frac1{r^2} \|\chi w\|_{\sigma} + \frac{\theta_0}{r} \|\chi \grad w\|_{\sigma} \right),\label{13}
\end{eqnarray}
where $\chi = \chi_{\spt(\eta)} $ and  $\theta_0 = \theta(r,z_0)$, which in turn implies
\begin{eqnarray}
\lefteqn{\iint_Q |\grad w|^2\, d\mu_{\sigma}dt + \iint_Q |\grad^2 w|^2\, d\mu_{\sigma+2}dt}\nonumber\\
&\lesssim & r^4  \iint_{\widehat Q} f^2\, d\mu_{\sigma}dt + \frac1{r^4} \iint_{\widehat Q} w^2 + r^2 \theta_0^2 |\grad w|^2\, d\mu_{\sigma}dt\label{14}
\end{eqnarray}
via \eqref{9} and Young's inequality
 We will show the argument for \eqref{13} for the leading order terms only. For instance, from the symmetry of $\L_{\sigma}$ and the fact that $|\rho\grad \eta|\lesssim \theta_0/r$, we deduce that
\[
\left| \int \eta w \L_{\sigma}(\rho\grad \eta\cdot\grad w)\, d\mu_{\sigma}\right|   \lesssim\frac{\theta_0}{r} \|\L_{\sigma}(\eta w)\|_{\sigma} \|\chi \grad w\|_{\sigma}.
\]
Similarly, by integration by parts we calculate
\[
\left|\int \eta w \rho \grad(\rho\grad \eta):\grad^2 w\, d\mu_{\sigma}\right| \lesssim  \frac{\theta_0^2}{r^2} \int |\grad(\eta w)||\grad w|\, d \mu_{\sigma} + \frac1{r^2}\int \eta|w| |\grad w|\,d\mu_{\sigma},
\]
and conclude observing that $|\grad(\eta w)| \lesssim \theta^{-1}_0r^{-1} |w| + |\grad w|$.
 The remaining terms of $\tilde f$ can be estimated similarly.

To gain control over the third order derivatives of $\eta w$, we test the equation  with $\L_{\sigma}(\eta w)$. With the help of the symmetry and nonnegativity properties of $\L_{\sigma}$, we obtain the estimate
\[
\iint |\grad\L_{\sigma}(\eta w)|^2\, d\mu_{\sigma+1}dt \le \iint \L_{\sigma}(\eta w)\tilde f\, d\mu_{\sigma}dt.
\]
We have to find suitable estimates for the inhomogeneity term. Because we have to make use of the previous bound \eqref{14}, we have to shrink the cylinders $Q$ and $\widehat Q$, such that the new function $\eta$ is supported in the set where the old $\eta$ was constantly one. We claim that
\begin{equation}\label{15}
\iint_Q |\grad \L_{\sigma} w|^2\, d\mu_{\sigma+1}dt
\lesssim r^2 \iint_{\widehat Q}f^2\, d\mu_{\sigma}dt + \frac1{r^6}  \iint_{\widehat Q} w^2 + r^2\theta_0^2 |\grad w| \, d\mu_{\sigma}dt .
\end{equation}
Again, we will on provide the argument for the leading order terms only. We use the symmetry of $\L_{\sigma}$, the bounds on derivatives of  $\eta $ and the scaling of $\rho$ (cf.\ \eqref{8b} and \eqref{8c}) to estimate
\begin{eqnarray*}
\lefteqn{\left| \int \L_{\sigma}(\eta w) \L_{\sigma}(\rho\grad \eta\cdot \grad w)\, d\mu_{\sigma}\right| }\\
&=& \left| \int \grad\L_{\sigma} (\eta w) \cdot \grad (\rho\grad \eta\cdot \grad w)\, d\mu_{\sigma+1} \right|\\
& \lesssim&  \|\grad \L_{\sigma}(\eta w)\|_{\sigma+1}\left(\frac{\theta_0}{r^2} \|\chi \grad w\|_{\sigma} + \frac{1}r\|\chi \grad^2 w\|_{\sigma+2}\right).
\end{eqnarray*}
Similarly, 
\[
\left|\int\L_{\sigma}(\eta w) \rho\grad (\rho\grad\eta): \grad^2 w\, d\mu_{\sigma}\right| \lesssim \frac1{r^2} \|\L_{\sigma}(\eta w)\|_{\sigma} \|\chi \grad^2 w\|_{\sigma+2}.
\]
The estimates of the remaining terms have a similar flavor. We deduce \eqref{15} with the help of Young's inequality and \eqref{14}.

The estimate \eqref{15} is beneficial as it allows to estimate $\tilde f$ in $L^2_{\sigma}$. This time, it is enough to study the term that involves the third-order derivatives of $w$. We rewrite $\L_{\sigma}(\rho\partial_i \eta \partial_i w)  = \rho \partial_i\eta \L_{\sigma}\partial_ i w  - 2\rho \grad(\rho\partial_i \eta)\cdot \grad\partial_i w + \L_{\sigma}(\rho\partial_i \eta) \partial_i w$ and $\L_{\sigma}\partial_i w = \partial_i \L_{\sigma} w  - z_i \laplace w - (\sigma+1)\partial_i w$, and estimate
\begin{eqnarray*}
\| \L(\rho\partial_i \eta\partial_i w)\|_{\sigma}& \lesssim &  \frac1{r\theta_0} \|\chi\L_{\sigma} \partial_i w\|_{\sigma+2} +\frac1{r^2}\|\chi\grad^2 w\|_{\sigma+2} + \frac{\theta_0}{r^3} \|\chi\grad w\|_{\sigma}\\
&\lesssim & \frac1{r} \|\chi\partial_i \L_{\sigma} w\|_{\sigma+1} +\frac1{r^2} \|\chi\grad^2 w\|_{\sigma+2} + \frac{\theta_0}{r^3}\|\chi\grad w\|_{\sigma}.
\end{eqnarray*}
Upon redefining $Q$ and $\hat Q$ as in the derivation of \eqref{15}, an application of \eqref{14} and \eqref{15} then yields
\[
\int_{\eps r^2}^{r^2}\| \L(\rho\partial_i \eta\partial_i w)\|_{\sigma}^2\, dt \lesssim
\iint_{\widehat Q} f^2\, d\mu_{\sigma}dt   + \frac1{r^8}
\iint_{\widehat Q} w^2 +r^2\theta_0^2 |\grad w|^2\, d\mu_{\sigma}dt.
\]
The remaining terms of $\tilde f$ can be estimated in a similar way. Applying the energy estimate from Lemma \ref{Lemma4} to the evolution equation for $\eta w$, we thus deduce
\begin{gather}
\iint_Q (\partial_t w)^2\, d\mu_{\sigma}dt + \iint_Q |\grad^2 w|^2\, d\mu_{\sigma}dt\hspace{9em}\nonumber\\ \mbox{}+ \iint_Q |\grad^3 w|^2\, d\mu_{\sigma+2}dt + \iint_Q |\grad^4 w|^2\,d \mu_{\sigma+4}dt\nonumber\\
 \hspace{9em} \lesssim \iint_{\widehat Q} f^2\, d\mu_{\sigma}dt + \frac1{r^8} \iint_{\widehat Q} w^2 + r^2 \theta_0^2 |\grad w|^2 \, d\mu_{\sigma}dt.\label{15a}
\end{gather} 
Notice that the above bound on the second order derivatives and \eqref{14} together imply that
\[
\frac{\theta_0^4}{r^4} \iint_Q |\grad^2 w|^2\, d\mu_{\sigma}dt \lesssim \iint_{\widehat Q} f^2\, d\mu_{\sigma}dt + \frac1{r^8} \iint_{\widehat Q} w^2 + r^2 \theta_0^2 |\grad w|^2 \, d\mu_{\sigma}dt.
\]
Similarly, we can produce the factor $\theta_0^2/r^2$ in front of the integral containing the third order derivatives. Indeed, because $\grad\L_{\sigma} w = \L_{\sigma+1} \grad w + z\laplace w - \grad^2 w z + (\sigma+1)\grad w$, the bound \eqref{15}  yields
\begin{eqnarray*}
\lefteqn{\iint_Q|\L_{\sigma+1} \grad w|^2 \,d\mu_{\sigma+1}dt}\\
&\lesssim &r^2\iint_{\widehat Q} f^2\, d\mu_{\sigma}dt + \frac1{r^6} \iint_{\widehat Q} w^2 + r^2 \theta_0^2 |\grad w|^2 \, d\mu_{\sigma}dt\\
&&\mbox{} + \iint_{\widehat Q} |\grad^2 w|^2\, d\mu_{\sigma+1}dt + \iint_{\widehat Q} |\grad w|^2\, d\mu_{\sigma+1}dt.
\end{eqnarray*}
The second order term on the right-hand side is controlled with the help of the Cauchy--Schwarz inequality, \eqref{14} and \eqref{15a}. The first-order term is of higher order as a consequence of \eqref{14}. It remains to invoke \eqref{9} to the effect that
\[
\iint_Q |\grad^3 w|^2\, d\mu_{\sigma+3}dt \lesssim r^2\iint_{\widehat Q} f^2\, d\mu_{\sigma}dt + \frac1{r^6} \iint_{\widehat Q} w^2 + r^2 \theta_0^2 |\grad w|^2 \, d\mu_{\sigma}dt.
\]
Combining the latter with \eqref{15a} yields the statement of the Lemma. 
\end{proof}

\subsection{Estimates for the homogeneous equation}\label{SS:hom}

In this subsection, we study the initial value problem for the homogeneous equation
\begin{equation}\label{15b}
\left\{\begin{array}{rcll}\partial_t w + \L_{\sigma}^2 w + n \L_{\sigma} w& = &0\qquad &\mbox{in }(0,\infty)\times B_1(0)\\
w(0,\tacka) &=&g\qquad &\mbox{in }B_1(0)\end{array}\right.
\end{equation}

Our first goal is a pointwise higher order regularity estimate.

\begin{lemma}\label{Lemma6}
Let $0<\eps<1$ and $0<\delta <1$ be given. Let $w$ be a solution to the homogeneous equation \eqref{15b}. If $ \eps,\delta\in(0,1)$ and $\delta$ is sufficiently small, then the following holds for any $z_0\in \overline{B_1(0)}$, $\tau\ge 0$ and $0<r\lesssim 1$:
\begin{eqnarray*}
\lefteqn{|\partial_t^k\partial_z^{\beta} w(t,z)|^2}\\
&\lesssim& \frac{r^{-8k-2|\beta|}\theta(r,z_0)^{-2|\beta|}}{r^4 |B_r^d(z_0)|_{\sigma}} \int_{\tau}^{\tau+r^4} \int_{B_r^d(z_0)} w^2 + r^2\theta(r,z_0)^2|\grad w|^2\, d\mu_{\sigma}dt,
\end{eqnarray*}
for any $(t,z)\in (\tau+ \eps r^4,\tau + r^4]\times B_{\delta r}^d(z_0)$.
\end{lemma}

\begin{proof}The lemma is a consequence of the local higher order regularity estimate
\begin{equation}
\label{16}
\iint_Q (\partial_t^k\partial_z^{\beta} w)^2\, d\mu_{\sigma} dt \lesssim r^{-8k-2|\beta|}\theta(r,z_0)^{-2|\beta|}\iint_{\widehat Q} w^2 +r^2\theta(r,z_0)^2|\grad w|^2\, d\mu_{\sigma}dt,
\end{equation}
where $Q$ and $\widehat Q$ are defined as in Lemma \ref{Lemma5}, and a Morrey estimate in the weighted space $L^2(\mu_{\sigma})$ (see, e.g., \cite[Lemma 4.9]{Seis15b}).   Notice that \eqref{16} is trivial for $(k,|\beta|)\in \{(0,1),(0,2)\}$. In the following, we write $\theta_0 = \theta(r,z_0)$.

To prove \eqref{16} for general choices of $k$ and $\beta$,  it is convenient to consider separately the two cases $\sqrt{\rho(z_0)}\lesssim r$ and $\sqrt{\rho(z_0)}\gg r$. The second case is relatively simple: Since $\rho\sim \rho(z_0)$ by \eqref{8c} in both $ Q$ and $\widehat Q$, we deduce \eqref{16} in the cases $(k,|\beta|)\in \{(1,0),(0,2),(0,3),(0,4)\}$ directly from Lemma \ref{Lemma5} (with $f=0$).  
To gain control on higher order derivatives, we differentiate with respect to $z_i$,
\begin{eqnarray*}
\partial_t \partial_i w + \L^2_{\sigma} \partial_i w +n\L_{\sigma}\partial_i w &=& -z_i \laplace \L_{\sigma} w  - (\sigma+1) \partial_i \L_{\sigma} w - \L_{\sigma} (z_i\laplace w)\\
&&\mbox{} -(\sigma+1) \L_{\sigma}\partial_i w - nz_i \laplace w - n(\sigma+1) \partial_i w.
\end{eqnarray*}
Denoting by $\tilde f$ the right-hand side of this identity and applying Lemma \ref{Lemma5} yields the estimate
\begin{eqnarray*}
\lefteqn{\iint_{Q}  |\grad^4 \partial_i w|^2\, d\mu_{\sigma+4}dt}\\
&\lesssim &\iint_{\widehat Q} \tilde f^2\, d\mu_{\sigma}dt + \frac1{r^8} \iint_{\widehat Q} (\partial_i w)^2 + r^2\theta_0^2 |\grad \partial_ i w |^2\, d\mu_{\sigma}dt.
\end{eqnarray*}
We invoke the previously derived bound and the fact that $\rho\sim \rho(z_0)$ to conclude the statement in the case $(k,|\beta|)=(0,5)$. Higher order derivatives are controlled similarly via iteration.

The proof in the case $\rho(z_0)\lesssim r$ is lengthy and tedious. As similar results have been recently obtained in \cite{John15,Kienzler16,Seis15b} and most the involved tools have been already applied earlier in this paper, we will only outline the argument in the following. Thanks to \eqref{8e}, it is enough to study the situation where $z_0\in \partial B_1(0)$, and upon shrinking $\delta$, we may assume that $\Phi$ constructed in Lemma \ref{Lemma2} is  a diffeomorphism from  $B_{\delta r}^2(z_0)$ onto a subset of the half space. Under $\Phi$, the homogeneous equation \eqref{15b} transforms into
\[
\partial_t \tilde w + \tilde \L_{\sigma}^2\tilde w +n \tilde \L_{\sigma} \tilde w = \tilde f,
\]
where $\tilde f$ is of higher order at the boundary. Because $\tilde \L_{\sigma}$ commutes with tangential derivatives $\tilde \partial_i$ for $i\in \{1,\dots,N-1\}$, control on higher order tangential derivatives are deduced from Lemma \ref{Lemma5}. To obtain control on vertical derivatives, we recall that $\tilde \partial_N \tilde \L_{\sigma} = \tilde \L_{\sigma+1} \tilde \partial_N - \tilde \laplace'$. Arguing as in the proof of Lemma \ref{Lemma3} gives the desired estimates. Again, bounds on higher order derivatives and mixed derivatives are obtained by iteration.
\end{proof}

For the proof of the Gaussian estimates and the Whitney measure estimates for the homogeneous problem, it is convenient to introduce a family of auxiliary functions $\chi_{a,b}: \overline{B_1(0)}\times \overline{B_1(0)} \to \R$, given by
\[
\chi_{a,b}(z,z_0) = \frac{a\hat d(z,z_0)^2}{\sqrt{b^2 +\hat d(z,z_0)^2}},
\]
where $a,b\in \R$ are given parameters, and 
\[
\hat d(z,z_0)^2 = \frac{|z-z_0|^2}{\sqrt{\rho(z)^2 + \rho(z_0)^2 + |z-z_0|^2}}\sim d(z,z_0)^2.
\]
It can be verified by a short computation that $\rho|\grad_z \hat d^2|^2\lesssim \hat d$ and $\rho |\grad^2_z\hat d^2|\lesssim 1$, with the consequence that
\begin{eqnarray}
\sqrt{\rho(z)} |\grad_z\chi_{a,b}(z,z_0)|&\lesssim& |a|,\label{17}
\\
\rho(z) |\grad^2_z\chi_{a,b}(z,z_0)|&\lesssim& \frac{|a|}{|b|},\label{18}
\end{eqnarray}
uniformly in $z,z_0\in \overline{B_1(0)}$. Because $\g$ is conformally flat with $\g\sim \rho^{-1}(dz)^2$, the gradient $\grad_{\g}$ on $(\M,\g)$ obeys the scaling $\grad_{\g}\sim \rho\grad$, and thus \eqref{17} can be rewritten as $\sqrt{\g(\grad_{\g}\chi,\grad_{\g}\chi)}\lesssim |a|$ (where we have dropped the indices and $z_0$). The latter implies that $\chi = \chi_{a,b}(\tacka,z_0)$ is Lipschitz with respect to the intrinsic topology, that is,
\begin{equation}
\label{19}
|\chi_{a,b}(z,z_0)- \chi_{a,b}(z',z_0)|\lesssim |a|d(z,z').
\end{equation}

We derive some new weighted energy estimates. 

\begin{lemma}\label{Lemma7}Let $w$ be the solution to the homogeneous equation \eqref{15b}. Let $a,b\in \R$ and $z_0\in\overline{B_1(0)}$ be given. Define $\chi = \chi_{a,b}(\tacka,z_0)$. Then there exists a constant $C>0$ such that for any $T>0$ it holds
\begin{gather*}
\sup_{[0,T]} \int e^{2\chi} w^2\, d\mu_{\sigma} +\int_0^T \int e^{2\chi}|\grad w|^2+ (\L_{\sigma}(e^{\chi}w))^2\, d\mu_{\sigma}dt\\
\lesssim  e^{C\left(\frac{a^2}{b^2}  + a^4\right)T} \int e^{2\chi}g^2\, d\mu_{\sigma},
\end{gather*}
%
\end{lemma}

\begin{proof}The quantity $e^{\chi}w$ evolves according to
\begin{eqnarray*}
\lefteqn{\partial_t (e^{\chi}w) + \L_{\sigma}^2 (e^{\chi}w) + n\L_{\sigma}(e^{\chi}w)}\\
&= & -2\rho\grad e^{\chi}\cdot \grad\L_{\sigma} w + \L_{\sigma} e^{\chi} \L_{\sigma}w -2\L_{\sigma} (\rho\grad e^{\chi}\cdot \grad w)\\
&&\mbox{} + \L_{\sigma} ((\L_{\sigma} e^{\chi}) w) -2n\rho \grad e^{\chi} \cdot \grad w + n(\L_{\sigma}e^{\chi}) w.
\end{eqnarray*}
Denoting the right-hand side by $\tilde f$ and testing with $e^{\chi}w$ yields
\begin{equation}
\label{22}
\frac{d}{dt}\frac12 \int (e^{\chi} w)^2\, d\mu_{\sigma} + \int (\L_{\sigma}(e^{\chi} w))^2\, d\mu_{\sigma}  + n \int |\grad (e^{\chi} w)|^2\, d\mu_{\sigma+1} = \int e^{\chi} w\tilde f\, d\mu_{\sigma},
\end{equation}
where we have used once more the symmetry of $\L_{\sigma}$.  We claim that the term on the right can be estimated as follows:
\begin{equation}
\label{20}
\int e^{\chi} w\tilde f\, d\mu_{\sigma}\lesssim\eps\left( \|\L_{\sigma}(e^{\chi}w )\|_{\sigma}^2  + \|\grad(e^{\chi}w)\|_{\sigma+1} \right) + \left(1 +  \frac{a^2}{b^2}  + a^4\right)\|e^{\chi} w\|_{\sigma}^2,
\end{equation}
where $\eps$ is some small constant that allows us to absorb the first two terms in the left-hand side of the energy estimate above.
Indeed, a multiple integrations by parts and the bounds \eqref{17} and \eqref{18} yield that the left-hand side of \eqref{22} is bounded by
\begin{gather*}
 |a| \|\L_{\sigma}\zeta\|_{\sigma}  \|\grad\zeta\|_{\sigma+1} + a^2 \|\grad\zeta\|_{\sigma+1}^2 + \left(|a|+\frac{a^2}{|b|}+|a|^3 \right)\|\grad \zeta\|_{\sigma+1} \|\zeta\|_{\sigma} \\
 +a^2 \|\grad\zeta\|_{\sigma+1}\|\zeta\|_{\sigma-1} + \left(\frac{|a|}{|b|} + a^2\right) \|\L_{\sigma}\zeta\|_{\sigma} \|\zeta\|_{\sigma}  + \left(1+ \frac{a^2}{b^2}  + a^4\right) \|\zeta\|_{\sigma}^2\\
 + \left(\frac{a^2}{|b|}+ |a|^3\right) \|\zeta\|_{\sigma-1}\|\zeta\|_{\sigma}  + |a|\|\L_{\sigma}\zeta\|_{\sigma} \|\zeta\|_{\sigma-1} + a^2 \|\zeta\|_{\sigma-1}^2
\end{gather*}
where we have set $\zeta = e^{\chi}w$. We next claim that
\begin{equation}
\label{21}
\|\zeta\|_{\sigma-1}\lesssim \|\zeta\|_{\sigma} + \|\grad\zeta\|_{\sigma+1}.
\end{equation}
Indeed,  recall the Hardy--Poincar\'e inequality
\[
\left\|\zeta - \avint \zeta\, d\mu_{ \tilde \sigma-1}\right\|_{ \tilde \sigma-1} \lesssim \|\grad \zeta\|_{ \sigma +1},
\]
cf.\ \cite[Lemma 3]{Seis14}, which holds true for any $\tilde \sigma\ge \sigma$, because  $\sigma>0$. In particular, $\|\zeta\|_{\sigma-1}\lesssim \|\grad \zeta\|_{\sigma+1}  + \left|\int \zeta\, d\mu_{\sigma-1}\right|$. Notice that for any $\alpha\in (0,\tilde \sigma)$, it holds that
\[
\left|\int \zeta\, d\mu_{\tilde \sigma-1}\right|
 = \left|\int \zeta\rho^{\alpha}\, d\mu_{\tilde \sigma-1-\alpha}\right| 
 \lesssim \left(\int \zeta^2\rho \,d\mu_{\tilde \sigma-1-\alpha}\right)^{1/2} \lesssim \|\zeta\|_{\tilde \sigma-1+\alpha},
\]
by Jensen's inequality because $\mu_{\tilde \sigma-1+\alpha}$ is a finite measure. Applying the previous two estimates iteratively yields \eqref{21}.
Hence, combining \eqref{21} and the interpolation inequality \eqref{23} with the bound on the inhomogeneity and using Young's inequality yields \eqref{20}.

Now \eqref{22} and \eqref{20} imply for $\eps$ sufficiently small that
\[
\frac{d}{dt}\int (e^{\chi} w)^2\, d\mu_{\sigma} + \int (\L_{\sigma}(e^{\chi}w ))^2\, \mu_{\sigma} \lesssim  \left(1 +  \frac{a^2}{b^2}  + a^4\right)\int (e^{\chi} w)^2\, d\mu_{\sigma}.
\]
In view of the bound \eqref{9} we have the estimate $
 \|\grad(e^{\chi}w)\|_{\sigma}\lesssim \|\L_{\sigma}(e^{\chi}w)\|_{\sigma}$. 
Therefore, invoking the product rule of differentiation
\[
\|e^{\chi} \grad w||_{\sigma}\le \| \grad(e^{\chi} w)\|_{\sigma} + \| e^{\chi} w\grad \chi\|_{\sigma}
\lesssim  \|\L_{\sigma}(e^{\chi}w)\|_{\sigma} + |a| \|e^{\chi} w\|_{\sigma-1}.
\]
Observe that \eqref{21} and \eqref{23} imply that
\[
|a| \|e^{\chi} w\|_{\sigma-1} \lesssim \left(|a|+a^2\right) \|e^{\chi} w\|_{\sigma} + \|\L_{\sigma} (e^{\chi} w)\|_{\sigma}.
\]
Combining the previous estimates with a Gronwall argument yields the statement of the lemma.
\end{proof}

The following estimate is a major step towards Gaussian estimates.
\begin{lemma}\label{Lemma8}
Let $w$ be the solution to the homogeneous equation \eqref{15b}. Let $a,b\in \R$ be given. Then there exists a constant $C>0$ such that for all $z, z_0 \in \overline{B_1(0)}$, $0<r\lesssim 1$, $t\in \left(\frac12 r^4, r^4\right)$, $k\in \N_0$ and $\beta\in \N_0^N$ it holds that
\[
|\partial_t^k\partial_z^{\beta} w(t,z)|
 \lesssim  \frac{r^{-4k -|\beta|} \theta(r,z)^{-|\beta|}}{|B_r^d(z)|^{1/2}_{\sigma}} e^{C\left(\frac{a^2}{b^2} + a^4\right)t - \chi_{a,b}(z,z_0)}
\|e^{\chi_{a,b}(\tacka,z_0)}g\|_{\sigma}.
\]

\end{lemma}

\begin{proof}
For abbreviation, we write $\chi = \chi_{a,b}(\tacka,z_0)$ and $\theta= \theta(r,z)$. From Lemma \ref{Lemma6} (with $z_0=z$ and $\tau=0$) we deduce the estimate
\begin{eqnarray}
\lefteqn{|\partial_t^k \partial_z^{\beta} w(t,z)|^2}\nonumber\\
&\lesssim& \frac{r^{-8k -2|\beta|} \theta^{-2|\beta|}}{r^4 |B_r^d(z)|_{\sigma}} \sup_{ B_r^d(z)}e^{-2\chi} \int_{0}^{r^4} \int_{B_r^d(z)} e^{2\chi} w^2 + r^2\theta^2 e^{2\chi}|\grad w|^2\, d\mu_{\sigma}dt,\hspace{2em}\label{28}
\end{eqnarray}
for all $t\in \left(\frac12 r^4, r^4\right)$.
We first observe that the Lipschitz estimate \eqref{19} implies that
\[
\sup_{ B_r^d(z)} e^{-\chi} \lesssim e^{ -\chi(z)+ a^4 r^4}.
\]
To estimate the integral expression in \eqref{28}, we distinguish the cases $\sqrt{\rho(z)}\le r$ and $\sqrt{\rho(z)}\ge r$.
In the first case, we we apply Lemma \ref{Lemma7} and obtain
\begin{gather*}
 r^4 \left(\sup_{[0,r^4]} \int e^{2\chi} w^2 \, d\mu_{\sigma} + \int_0^{r^4} \int e^{2\chi} |\grad w|^2\, d\mu_{\sigma}dt\right)\\
\lesssim   r^4 e^{C \left(\frac{a^2}{b^2}  + a^4\right)r^4 } \int e^{2\chi}g^2\, d\mu_{\sigma}
\end{gather*}
for some $C>0$. In the second case, we only focus on second term, i.e., the gradient term. The argument for the first term remains unchanged. Because $\rho\sim \rho(z)$ in the domain of integration (cf.\ \eqref{8c}), it holds that
\begin{eqnarray*}
\lefteqn{\int_0^{r^4} \int_{B_r^d(z)} r^2\theta^2  e^{2\chi}|\grad w|^2\,d\mu_{\sigma}dt}\\
&\lesssim& r^2 a^2 \int_0^{r^4} \int e^{2\chi} w^2\, d\mu_{\sigma} dt + r^2 \int_0^{r^4} \int |\grad(e^{\chi} w)|^2\, d\mu_{\sigma+1}dt,
\end{eqnarray*}
where we have used \eqref{17}. By using \eqref{23} and Young's inequality, we further estimate
\[
r^2 \int |\grad(e^{\chi}w)|^2\, d\mu_{\sigma+1} \lesssim  r^4\int (\L_{\sigma}(e^{\chi}w))^2\, d \mu_{\sigma}  + \int (e^{\chi} w)^2\, d\mu_{\sigma},
\]
which in turn yields
\[
\int_0^{r^4} \int_{B_r^d(z)} r^2\theta^2  e^{2\chi}|\grad w|^2\,d\mu_{\sigma}dt
\lesssim r^4\left(1+r^2a^2\right) e^{C\left(\frac{a^2}{b^2} + a^4\right)r^4}\int e^{2\chi}g^2\, d\mu_{\sigma}
\]
via Lemma \ref{Lemma7}. Notice that we can eliminate the factor $r^2a^2$ in the previous expression upon enlarging the constant $C$. Substituting the previous bounds into \eqref{28} yields the statement of the lemma.
%
%
\end{proof}

For large times, we have exponential decay as established in the lemma that follows.

\begin{lemma}\label{Lemma9}
Let $w$ be the solution of the initial value problem \eqref{7} with $f=0$. Then 
for any $k\in \N_0$, $\beta\in \N_0^N$, $t\ge\tfrac12$ and $z\in \overline{B_1(0)}$ it holds that
\[
\left|\partial_t^k\partial_z^{\beta}\left( w(t,z) - \avint g\, \mu_{\sigma}\right)\right|\lesssim e^{-\lambda_1 t} \|\grad g\|_{\sigma+1}.
\]
\end{lemma}

\begin{proof}The proof is an easy consequence of Lemma \ref{Lemma6} and a spectral gap estimate for $\L_{\sigma}$. Indeed,  applying Lemma \ref{Lemma6} with $t = \tau+\frac14$, $\eps=\frac14$, $r=1$ and $\tau\ge \frac14$ to $w-c$, where $c = \avint w\, d\mu_{\sigma}$ is a constant of the evolution, we obtain the estimate
\[
|\partial_t^k\partial_z^{\beta}\left(w(t,z)-c\right) |\lesssim \int_{t-\frac14}^{t+\frac34} \int (w-c)^2 +|\grad w|^2\, d\mu_{\sigma}dt.
\]
Thanks to the Hardy--Poincar\'e inequality \cite[Lemma 3]{Seis14} and because $\mu_{\sigma+1}\lesssim \mu_{\sigma} $, we can drop the term $(w-c)^2$ in the integrand. To prove the statement of the Lemma, we thus have to establish the estimate
\begin{equation}
\label{25}
\int_{t-\frac14}^{t+\frac34} |\grad w|^2\, d\mu_{\sigma} dt \lesssim e^{-2\lambda_1 t} \int |\grad g|^2\, d\mu_{\sigma+1}.
\end{equation}
For this purpose, we 
test the homogeneous equation with $w$ and invoke the symmetry and nonnegativity properties of $\L_{\sigma}$ and obtain the energy estimate
\[
\frac{d}{dt}\frac12 \int |\grad w|^2\, d\mu_{\sigma+1} +\int( \L_{\sigma} w)^2\, d\mu_{\sigma} \le 0.
\]
On the one hand, integration in time over $ \left[t-\frac14,t+\frac34\right]$ and the a priori estimate \eqref{9} yield
\begin{equation}
\label{24}
\int_{t-\frac14}^{t+\frac34} |\grad w|^2\, d\mu_{\sigma} dt \lesssim \int |\grad w\left( t-\tfrac14\right)|^2\, d\mu_{\sigma+1}.
\end{equation}
On the other hand, the smallest non-zero eigenvalue $\lambda_1$ of $\L_{\sigma}$ yields the spectral gap estimate
\[
\int (\L_{\sigma}w)^2\, d\mu_{\sigma} = \int \grad w\cdot\grad\L_{\sigma} w\, d\mu_{\sigma+1} \ge \lambda_1 \int |\grad w|^2\, d\mu_{\sigma+1},
\]
which we combine with the energy estimate from above to get
\[
\int |\grad w\left(t-\tfrac14\right)|^2\, d\mu_{\sigma+1} \lesssim  e^{-2\lambda_1 t} \int |\grad g|^2\, d\mu_{\sigma+1}.
\]
Plugging this estimate into \eqref{24} yields \eqref{25} as desired.
\end{proof}

We are now in the position to prove the desired maximal regularity estimate for the homogeneous problem. Let us start with the latter.

\begin{prop}\label{Prop1}
Let $w$ be the solution to the homogeneous equation \eqref{15b}. Then
\[
\|w\|_{L^{\infty}} \lesssim \|g\|_{L^{\infty}} 
\]
and
\[
\|w\|_{X(p)} + \|\grad w\|_{L^{\infty}}\lesssim \|\grad g\|_{L^{\infty}}.
\]
\end{prop}

\begin{proof}Thanks to the exponential decay estimates from Lemma \ref{Lemma9}, it is enough to focus on the  norms for small times, $T\le1$. We fix $z_0\in \overline{B_1(0)}$ for a moment and let $r\lesssim 1$ and $t\in\left(\tfrac12 r^4,r^4\right)$ be arbitrarily given. As before, we set $\theta_0 = \theta(r,z_0)$. Because $w -g(z_0)$ is a solution to the homogeneous equation with initial value $g-g(z_0)$, an application of Lemma \ref{Lemma8} with $a=-\frac1r$ and $b=r$ yields the estimate
\begin{equation}\label{27}
|\left.\partial_t^k\partial_z^{\beta}\right|_{z=z_0} \left(w(t,z)-g(z_0)\right)|
 \lesssim \frac{r^{-4k - |\beta|}\theta_0^{-|\beta|}}{|B_r^d(z_0)|_{\sigma}^{1/2}} \|e^{\chi_{-\frac1r,r}(\tacka,z_0)} (g-g(z_0)) \|_{\sigma}.
\end{equation}
Notice that the function $\chi$ drops out in the exponential prefactor because $\chi(z_0,z_0)=0$. We claim that
\begin{equation}
\label{26}
\|e^{\chi_{-\frac1r,r}(\tacka,z_0)} (g-g(z_0)) \|_{\sigma}\lesssim \min\left\{  \|g\|_{L^{\infty}},
r\theta_0 \|\grad g\|_{L^{\infty}}\right\}|B_r^d(z_0)|_{\sigma}^{1/2}
\end{equation}
The proof of this estimate has been already displayed earlier, see, e.g., Proof of Proposition 4.2 in \cite{Seis14}. For the convenience of the reader, we recall the simple argument. Notice first that $|g(z) - g(z_0)|\lesssim \min\{\|g\|_{L^{\infty}}, |z-z_0|\|\grad g\|_{L^{\infty}}\}$ .
On every annulus $A_j = B_{jr}^{\hat d}(z_0)\setminus B_{(j-1)r}^{\hat d} (z_0)$ it holds that $\chi_{-\frac1r,r}(z,z_0)\le -\frac{j-1}{\sqrt2} $ as can be verified by an elementary computation, and thus, for $s\in\{0,1\}$, we have
\[
\int_{A_j} e^{2\chi_{-\frac1r,r}(z,z_0)}|z-z_0|^{2s} \, d\mu_{\sigma}(z) \lesssim j^{2s} r^{2s} \theta(jr,z_0)^{2s} e^{-\sqrt2j }|A_j|_{\sigma}
\]
as a consequence of \eqref{8d}. Clearly, $\theta(jr,z_0)\le j\theta_0$. We notice that $A_j = \emptyset$ for each $j\gg\frac1r$. On the other hand, thanks to the volume formula \eqref{8f}, it holds 
\[
|A_j|_{\sigma} \lesssim j^{2(N+\sigma)} |B_r^d(z_0)|_{\sigma}.
\]
It remains to notice that the annuli $\{A_j\}_{j\in\N}$  cover $\overline{B_1(0)}$ and deduce that
\[
\|e^{\chi_{-\frac1r,r}(\tacka,z_0)}|\tacka - z_0|^s \|_{\sigma}\lesssim (r\theta_0)^s |B_r^d(z_0)|_{\sigma}^{1/2} \left(\sum_{j\in \N}  e^{-\sqrt2 j }j^{\kappa}\right)^{1/2} ,
\]
for some $\kappa = \kappa(s)>0$.
Because the series is convergent, we have thus proved the bound in \eqref{26}.

We now combine \eqref{27} and \eqref{26} to the effect of 
\[
r^{4k +|\beta|} \theta_0^{|\beta|}|\left.\partial_t^k\partial_z^{\beta}\right|_{z=z_0} \left(w(t,z)-g(z_0)\right)| \lesssim  \| g\|_{L^{\infty}}
\]
and
\[
r^{4k+|\beta| -1}\theta_0^{|\beta|-1}|\left.\partial_t^k\partial_z^{\beta}\right|_{z=z_0} \left(w(t,z)-g(z_0)\right)| \lesssim  \|\grad g\|_{L^{\infty}}.
\]
We obtain the uniform bounds on $w$ and $\grad w$ in the time interval $[0,1]$ by setting $(k,|\beta|) = (0,0)$ in the first and $(k,|\beta|) = (0,1)$ in the second estimate. (Recall that we use Lemma \ref{Lemma9} to extend the estimates to times $t\ge 1$.)  To control in $X(p)$, we choose $(k,|\beta|)\in \{(1,0),(0,2),(0,3),(0,4)\}$, raise the second of the above estimates to the power $p$ and average over $Q_r(z)$. For instance,  if $(k,|\beta|)  =(0,2)$, this leads to
\[
\frac{r^p}{|Q_r^d(z)|}\iint_{Q_r^d(z)} \theta(r,z_0)^p |\grad^2 w(t,z_0)|^p\, dz_0dt \lesssim \|\grad g\|_{L^{\infty}}^p.
\]
If view of \eqref{8b} and \eqref{8c}, it holds that $\theta(r,z_0)\sim \theta(r,z)$ uniformly in $B_r^d(z)$, and thus, from maximizing in $r$ and $z$ we obtain
\[
\sup_{\substack{ z\in \overline{B_1(0)} \\0<r\lesssim 1}} r \theta(r,z) |Q_r^d(z)|^{-\frac1p} \|\grad^2 w\|_{L^p(Q_r^d(z))}\lesssim \|\grad g\|_{L^{\infty}}.
\]
Higher order derivatives are bounded analogously.
\end{proof}

Gaussian estimates are contained in the following statement.

\begin{prop}\label{Prop2}
There exists a unique function $G:(0,\infty)\times \overline{B_1(0)}\times \overline{B_1(0)}\to \R$ with the following properties:
\begin{enumerate}
\item If $w$ is the solution to the homogeneous equation \eqref{15b}, then for any $k\in \N_0$, $\beta\in \N_0^N$ and $(t,z)\in (0,\infty)\times \overline{B_1(0)}$
\[
\partial_t^k\partial_z^{\beta} w(t,z) = \int \partial_t^k\partial_z^{\beta} G(t,z,z') g(z')\, d\mu_{\sigma}.
\]
\item The function $G$ is symmetric in the last two variables, that is,
\[
G(t,z,z') = G(t,z',z)
\]
for all  $(t,z,z')\in (0,\infty)\times \overline{B_1(0)}\times\overline{B_1(0)}$.
\item For any $z'\in B_1(0)$, $G' = G(\tacka, \tacka, z')$ solves the homogeneous equation
\[
\partial_t G' + \L_{\sigma}^2 G' + \L_{\sigma} G' = 0.
\]
Moreover,
\[
\rho^{\sigma} G' \stackrel{t\downarrow 0}{\longrightarrow } \delta_{z'} \quad\mbox{in the sense of distributions.}
\]
\item It holds that
\[
|\partial_t^k\partial_z^{\beta} G(t,z,z')|\lesssim \frac{ \sqrt[4]{t}^{\, -4k-|\beta|}\theta(\sqrt[4]{t},z)^{-|\beta|}}{|B_{\sqrt[4]{t}}^d(z)|_{\sigma}^{1/2}|B_{\sqrt[4]{t}}^d(z')|_{\sigma}^{1/2}} e^{-C \left(\frac{d(z,z')}{\sqrt[4]{t}}\right)^{4/3}},
\]
for all  $(t,z,z')\in (0,1]\times \overline{B_1(0)}\times\overline{B_1(0)}$ and  any $k\in\N_0$ and $\beta\in \N_0^N$.
\item It holds that
\[
|\partial_t^k \partial_z^{\beta} \left(G(t,z,z') - |B_1(0)|_{\sigma}^{-1}\right)|\lesssim e^{-\lambda_1 t}
\]
for all  $(t,z,z')\in [1,\infty)\times \overline{B_1(0)}\times\overline{B_1(0)}$ and  any $k\in\N_0$ and $\beta\in \N_0^N$.
\end{enumerate}
\end{prop}

The estimates in the fourth statement are usually referred to as ``Gaussian estimates''.

\begin{remark}\label{Remark1}
In the fourth statement we may freely interchange the balls centered at $z$ by balls centered at $z'$ and vice versa. Likewise, we can substitute $\theta(\sqrt[4]{t},z)$ by  $\theta(\sqrt[4]{t},z')$. This is a consequence of \eqref{8g}.
\end{remark}

The proof of this Proposition is (almost) exactly the one of \cite[Proposition 4.3]{Seis14}. We display the argument for completeness and the convenience of the reader.

\begin{proof}We first notice that the linear mapping $L_{\sigma}^2 \ni g\mapsto \partial_t^k\partial_z^{\beta} w(t,z)\in \R$ is bounded for every fixed $(t,z)\in (0,\infty)\times \overline{B_1(0)}$ and $(k,\beta) \in \N_0\times \N_0^N$. Indeed, for small times, boundedness is a consequence of Lemma \ref{Lemma8} (with $a=0$), and for large times, boundedness follows from successively applying Lemma \ref{Lemma9} and Lemma \ref{Lemma8} (with $a=0$), namely  $|\partial_t^k \partial_z^{\beta} w(t,z)|\lesssim \|w\left(\frac12\right)\|_{\sigma} + \|\grad w\left(\frac12\right)\|_{\sigma+1}  \lesssim \|g\|_{\sigma}$. Riesz' representation theorem thus provides us with the existence of a unique function $G_{k,\beta}(t,z,\tacka)\in L^2_{\sigma}$ such that
\[
\partial_t^k \partial_z^{\beta} w(t,z) = \int G_{k,\beta} (t,z,z') g(z')\, d\mu_{\sigma}(z').
\]
Setting $G = G_{0,0}$, uniqueness implies that $G_{k,\beta} = \partial_t^k\partial_z^{\beta} G$. Notice that $G$ inherits the symmetry in $z$ and $z'$ from the symmetry of the linear operator $\L^2_{\sigma}+\L_{\sigma}$ via the symmetry of the associated semi-group operator $e^{(\L_{\sigma}^2 + \L_{\sigma})t}$. 

We now turn to the proof of the Gaussian estimates. We shall write $\chi = \chi_{a,b}(\tacka,z_0)$ for some fixed $z_0\in\overline{B_1(0)}$ and set $\theta = \theta(r,z)$. We first notice that by Lemma \ref{Lemma8}, for $r\sim \sqrt[4]{t}$, we have
\[
|B_r^d(z)|^{1/2}_{\sigma} e^{\chi(z)} |w\left(\frac{t}2,z\right)| \lesssim e^{C\left(\frac{a^2}{b^2}  + a^4\right) t} \|e^{\chi} g\|_{\sigma},
\]
and thus, the mapping $\A$ defined by 
\[
(\A h)(z) = |B_r^d(z)|_{\sigma}^{1/2} e^{\chi(z)} \int G\left(\frac{t}2,z,z'\right) e^{-\chi(z')} h(z')\, d\mu_{\sigma}(z') ,
\]
for $z\in \overline{B_1(0)}$, is a bounded linear mapping from $L_{\sigma}^2$ to $L^{\infty}$ with
\[
\|\A\|_{L^2_{\sigma}\to L^{\infty}} \lesssim e^{C\left(\frac{a^2}{b^2}  + a^4\right)t}.
\]
By the symmetry of the Green's function, it holds that
\begin{eqnarray*}
\int \A h\xi\, d\mu_{\sigma} &=& \iint |B_r^d(z)|^{1/2}_{\sigma} e^{\chi(z)} G\left(\frac{t}2,z',z\right)  e^{-\chi(z')}h(z') \, d\mu_{\sigma}(z')d\mu_{\sigma}(z)\\
&=& \int e^{-\chi }w_{\xi}\left(\frac{t}2\right) h\, d\mu_{\sigma},
\end{eqnarray*}
if $w_{\xi}$ denotes the solution to the homogeneous equation with initial value $g_{\xi} = |B_r^d(\tacka)|_{\sigma}^{1/2} e^{\chi} \xi$, and if  $\xi\in L^1_{\sigma}$ is such that $g_{\xi}\in L^2_{\sigma}$. In particular, the action of the dual $\A^*: (L^{\infty})^*\to L^2_{\sigma}$ on such functions $\xi$ is given by $\A^* \xi  = e^{-\chi} w_{\xi}\left(\frac{t}2\right)$. Because $\|A\|_{L_{\sigma}^2\to L^{\infty}} = \|\A^*\|_{(L^{\infty})^* \to L_{\sigma}^2}$, we then have the estimate
\[
\|e^{-\chi} w_{\xi}\left(\frac{t}2\right)\|_{\sigma} \lesssim e^{C\left(\frac{a^2}{b^2}  + a^4\right)t} \|\xi\|_{L^1_{\sigma}}.
\]
An application of Lemma \ref{Lemma8} with $a$ replaced by $-a$ then yields that
\begin{eqnarray*}
\lefteqn{\left|\int \partial_t^k\partial_z^{\beta}G(t,z,\tacka) |B_r^d(\tacka)|_{\sigma}^{1/2} e^{\chi} \xi\, d\mu_{\sigma}\right|}\\
& \lesssim &\frac{r^{-4k-|\beta|}\theta^{-|\beta|}}{|B_r^d(z)|_{\sigma}^{1/2}} e^{C\left(\frac{a^2}{b^2} + a^4\right)t + \chi(z)} \|\xi\|_{L^1_{\sigma}}.
\end{eqnarray*}
By approximation, it is clear that this estimate holds for any $\xi\in L^1_{\sigma}$. Thanks to the duality $L^{\infty} = (L^1_{\sigma})^*$, we thus have
\[
\left|\partial_t^k\partial_z^{\beta}G(t,z,z')\right| \lesssim \frac{r^{-4k-|\beta|}\theta^{-|\beta|}}{|B_r^d(z)|_{\sigma}^{1/2} |B_r^d(z')|_{\sigma}^{1/2}} e^{C\left(\frac{a^2}{b^2}  + a^4\right)t + \chi(z) - \chi(z')}
\]
The term $-\chi(z')$ drops out of the exponent upon choosing $z'=z_0$. To conclude the argument for the Gaussian estimates, we distinguish two cases: First, if $\sqrt[4]{t} \ge d(z,z_0)$, then
\[
1\lesssim e^{-C\left(\frac{d(z,z_0)}{\sqrt[4]{t}}\right)^{4/3}},
\]
and thus the statement follows with $a=0$. Otherwise, if $\sqrt[4]{t} \le d(z,z_0)$, we choose $a = -\ell$ for some $\ell>0$ and $b\sim d = d(z,z_0)$ so that the exponent becomes
\[
\left(\frac{\ell^2}{d^2} +\ell^4\right)t  - \ell d
\]
modulo constant prefactors. We optimize the last two terms in $\ell$ by choosing $\ell\sim (d/t)^{1/3}$. It is easily checked that the exponent is bounded by an expression of the form $1  - (d/\sqrt[4]{t})^{4/3}$, which yields the desired result.

The remaining properties are immediate consequences of the preceding analysis.
\end{proof}

\subsection{Calder\'on--Zygmund estimates}\label{SS:CZ}

We will see at the beginning of the next subsection that the kernel representation of solutions of the homogeneous problem caries over to the ones of the inhomogeneous problem. This observation is commonly referred to as Duhamel's principle. To study regularity in the inhomogeneous problem the detailed knowledge of the Gaussian kernel provided by Proposition \ref{Prop2} is very helpful. A major step in the analysis of Whitney measures is the translation of the energy estimates from weighted $L^2$ to standard $L^p$ spaces. We are thus led to the study of singular integrals in the spirit of Calder\'on and Zygmund and the theory of Muckenhoupt weights.

Out of the Euclidean setting, a good framework for these studies is provided by spaces of homogeneous type, see Coifman and Weiss \cite{CoifmanWeiss71}, which are metric measure spaces, i.e., metric spaces endowed with a doubling Borel measure.\footnote{In fact, Coifman and Weiss introduced the notion of spaces of homogeneous type with quasi metrics instead of metrics.} The theory of singular integrals in spaces of homogeneous type was elaborated by Koch \cite{Koch99,Koch04,Koch08}. For the Euclidean theory, we refer to Stein's monographs \cite{Stein70,Stein93}.

Let us recall some pieces of the abstract theory. Let $(X,D)$ be a metric space endowed with a doubling Borel measure $\mu$.
A linear operator $T$ on $L^q(X,\mu)$ with $q\in(1,\infty)$ is called a \emph{Calder\'on--Zygmund operator} if $T$ can be written as
\[
T f(x) = \int_X K(x,y) f(y)\, d\mu(y)
\]
for all $x\in (\spt f)^c$ and $f\in L^{\infty}(X,\mu) \cap L^q(X,\mu)$, where $K : X\times X\to \R$ is a measurable kernel such that
\begin{align*}
y &\mapsto K(x,y) \in L^1_{\loc}(X\setminus\{x\},\mu),\\
x &\mapsto K(x,y) \in L^1_{\loc}(X\setminus\{y\},\mu),
\end{align*}
and satisfying the following boundedness and \emph{Calder\'on--Zygmund cancellation conditions}:
\begin{equation}
\label{29}
\sup_{x\not =y}V(x,y)|K(x,y)| \lesssim 1
\end{equation}
 and
\begin{equation}
\label{30}
\sup_{x\not = y}\sup_{x'\not = y'}V(x,y)\wedge V(x',y')|K(x,y) - K( x',y')| \lesssim  \left(\frac{D(x,x') + D(y, y')}{D(x,y) + D(x',y')}\right)^{\delta}
\end{equation}
for some $\delta \in (0,1]$.
Here we have used the notation
\[
V(x,y) = \mu\left(B^D_{D(x,y)}\left((x+y)/2\right)\right).
\]
It is worth noting  that the doubling property of $\mu$ implies that we could  equivalently have chosen to center the above balls at $x$ or $y$.

Finally, we call $\omega$ a \emph{$p$-Muckenhoupt weight} if 
\[
\sup_B \left( \frac1{\mu(B)}\int_B \omega\, d\mu\right)\left(\frac1{\mu(B)}\int_B \omega^{-\frac1{p-1}}\, d\mu\right)^{p-1} <\infty.
\]
The class of $p$-Muckenhoupt weights is denoted by $A_p(X,D,\mu)$. 

The theory of singular integrals asserts that any Calder\'on--Zygmund operator $T$ extends to a bounded operator on any $L^p(X,\mu)$ with $p\in(1,\infty)$, i.e.,
\[
\|T f\|_{L^p(\mu)} \lesssim \|f\|_{L^p(\mu)}.
\]
Moreover, if $\omega \in A_p$ is a Muckenhoupt weight, then $T$ is also bounded on $L^p(\mu\mres\omega)$, where $d(\mu\mres\omega) = \omega d\mu$.

In order to establish $L^p$ maximal regularity estimates for our problem at hand, we have to study singular integrals of the form
\[
T_{\ell,k,\beta} f(t,z) = \int_0^{\infty} \int K_{\ell,k,\beta}((t,z),(t',z')) f(t',z')\, d\mu_{\sigma}(z')dt',
\]
where $K_{\ell,k,\beta}((t,z),(t',z')) = \chi_{(0,t)}(t')\rho(z)^{\ell}\partial_t^k\partial_z^{\beta} G(t-t',z,z')$. In fact, we will see that $T_{\ell,k,\beta}$ is a Calder\'on--Zygmund operator on the product space $(0,\infty)\times B_1(0)$ provided that $\ell$, $k$, and $\beta$ are such that
\begin{equation}\label{34}
(\ell,k,|\beta|) \in \E=\left\{ (0,1,0), (0,0,2), (1,0,3),(2,0,4)\right\}.
\end{equation}
We will accordingly refer to any tuple $(\ell,k,\beta)$ in the above class as a \emph{Calder\'on--Zygmund exponent}. 

The product space $X = (0,\infty)\times B_1(0)$ will be endowed with the metric 
\[
D((t,z),(t',z')) = \sqrt[4]{|t-t'| + d(z,z')^4},
\]
which reflects the parabolic scaling of the linear differential operator, and the product  measure $\mu = \lambda^1\otimes \mu_{\sigma}$, with $\lambda^1$ denoting the one-dimensional Lebesgue.
Because $d$ is doubling, so is $D$, and thus the  metric measure space $(X,
D,\mu)$  is  \emph{of homogeneous type} in the sense of Coifman and Weiss \cite{CoifmanWeiss71} and is thus suitable for Calder\'on--Zygmund theory. Notice also that the volume tensor $V((t,z),(t',z'))$ simplifies to
\begin{equation}
\label{35}
V((t,z),(t',z')) \sim D((t,z),(t',z'))^4 |B_{D((t,z),(t',z'))}^d((z+z')/2)|_{\sigma}.
\end{equation}

Without proof, we state the following lemma:
\begin{lemma}\label{Lemma12}
If $(\ell,k,\beta)$ is such that  \eqref{34} holds, then $T_{\ell,k,\beta}$ is a Calder\'on--Zygmund operator.
\end{lemma}
%

The proof is almost identical to the one in the porous medium setting, see Lemmas 4.20 and 4.21 in \cite{Seis15b}. We will thus refrain from repeating the argument and refer the interested reader to the quoted paper.

\subsection{The inhomogeneous problem}\label{SS:inhom}

In this subsection, we consider the inhomogeneous problem with zero initial datum,
\begin{equation}\label{32}
\left\{\begin{array}{rcll}\partial_t w + \L_{\sigma}^2 w + n \L_{\sigma} w& = &f\qquad &\mbox{in }(0,\infty)\times B_1(0)\\
w(0,\tacka) &=&0\qquad &\mbox{in }B_1(0)\end{array}\right.
\end{equation}

Our first observation is that the kernel representation from Proposition \ref{Prop2} carries over to the inhomogeneous setting.

\begin{lemma}[Duhamel's principle]\label{Lemma10}
If $f\in L^2(L^2_{\sigma})$ and $w$ is the solution to the initial value problem \eqref{32}, then
\[
w(t,z) = \int_0^t \int G(t-t',z,z')f(t',z')\, d\mu_{\sigma}(z')dt'
\]
for all $(t,z)\in (0,\infty)\times \overline{B_1(0)}$.
\end{lemma}

\begin{proof}
The statement follows from the fact that $G$ is a fundamental solution, see statement 3 of Proposition \ref{Prop2}.
\end{proof}

\begin{prop}\label{Prop3}
Let $w$ be the solution to the initial value problem \eqref{32}. Then, for any $p\in(1,\infty)$ it holds
\begin{equation}
\label{33}
\|\partial_t w\|_{L^p} + \|\grad w\|_{L^p}  + \|\grad^2 w\|_{L^p} + \|\rho\grad^3 w\|_{L^p} + \|\rho^2\grad^4 w\|_{L^p}\lesssim \|f\|_{L^p}.
\end{equation}
\end{prop}

\begin{proof}The purpose of this lemma is to carry the energy estimates from Lemma \ref{Lemma4} over to the standard $L^p$ setting. This is achieved by applying the abstract theory recalled in the previous subsection. In fact, as a consequence of Lemma \ref{Lemma10}, any function  $\rho^\ell \partial_t^k\partial_z^{\beta} w$ has the kernel representation
\[
T_{\ell,k,\beta} f(t,z)  = \int_0^{\infty} \int K_{\ell,k,\beta}((t,z),(t',z')) f(t',z')\, d\mu_{\sigma}(z')dt',
\]
where 
\[
K_{\ell,k,\beta}((t,z),(t',z'))  = \chi_{(0,t)}(t')\rho(z)^{\ell}\partial_t^k\partial_z^{\beta}G(t-t',z,z').
\]
If $(\ell,k,\beta)$ are Calder\'on--Zygmund exponents \eqref{34}, by Lemma \ref{Lemma12}, the energy estimates from Lemma \ref{Lemma4} carry over to the any $L^p(L^p(\mu_{\sigma}))$ space with $p\in(1,\infty)$. Moreover, if $\nu$ is a Muckenhoupt weight in $A_p(B_1,d,\mu_{\sigma})$, then the operators $T_{\ell,k,\beta}$ are bounded on $L^p(L^p(\mu_{\sigma}\mres \nu))$. Notice that this is the case for weights of the form $\nu = \rho^{\gamma}$ precisely if $-(\sigma+1) <\gamma < (p-1)(\sigma+1)$. In particular, choosing $\gamma = -\sigma$, we see that  $T_{\ell,k,\beta}$ is bounded on $L^p = L^p(L^p)$ for any $p\in(1,\infty)$   because $\sigma>0$. This is the  statement of the proposition apart from the term $\|\grad w\|_{L^p}$. The control of this term can be deduced, for instance, from the analogous estimates for the porous medium equation, see Proposition 4.23 in \cite{Seis15b}, applied to $\partial_t w + n\L_{\sigma} w = f - \L_{\sigma}^2w$. This concludes the proof.
 \end{proof}

In the following, we consider the larger cylinders 
\[
\widehat Q_r^d(z_0) : = \left(\frac{r^4}4,r^4\right)\times B_{2r}^d(z_0)\quad\mbox{and}\quad \widehat Q(T)  = \left(\frac{T}4,T\right)\times \overline{B_1(0)}.
\]

\begin{lemma}\label{Lemma11}
\begin{enumerate}
\item   Suppose that $\spt f\subset \widehat Q_r^d(z_0)$ for some $z_0\in \overline{B_1(0)}$ and $0<r\lesssim 1$. Then for any $(\ell,k,\beta)$ satisfying \eqref{34} and any $p\in(1,\infty)$, it  holds that
\[
r^4 |Q_r^d(z_0)|^{-\frac1p} \|\rho^{\ell}\partial_t^k\partial_z^{\beta}  w\|_{L^p(Q_r^d(z_0))} \lesssim \|f\|_{Y(p)}.
\]
\item Suppose that $\spt f\subset \widehat Q(T)$ for some $T\ge 1$.  Then it holds for any $p\in(1,\infty)$ that
\[
\sum_{(\ell,k,|\beta|)\in\E} T\|\rho^{\ell}\partial_t^k\partial_z^{\beta}  w\|_{L^p(Q(T))} \lesssim \|f\|_{Y(p)}.
\]
\end{enumerate}
\end{lemma}

\begin{proof}
We will only prove the first statement. The argument for the second one is very similar. The desired estimate is an immediate consequence of Proposition \ref{Prop3}. Indeed, the latter implies that
\[
\|\rho^{\ell}\partial_t^k\partial_z^{\beta}  w\|_{L^p(Q_r^d(z_0))} \lesssim \| f\|_{L^p(\widehat Q_r^d(z_0))}.
\]
If now $\left\{Q_{r_i}^d(z_i)\right\}_{i\in I}$ is a finite cover of $\widehat Q_r^d(z_0)$ with radii $r_i\sim r$ and such that $\sum_{i} |Q_{r_i}^d(z_i)|\lesssim |\widehat Q_r^d(z_0)|$, then
\[
\| f\|_{L^p(\widehat Q_r^d(z_0))} \le \sum_{i\in  I} \| f\|_{L^p(Q_{r_i}^d(z_i))} \lesssim \frac1{r^4} |\widehat Q_r^d(z_0)|^{\frac1p} \| f\|_{Y(p)}.
\]
Notice that $|\widehat Q_r^d(z_0)| \lesssim |Q_r^d(z_0)|$, because $\mu_{\sigma}\otimes \lambda^1$ is doubling, which concludes the proof.
\end{proof}

In view of the definition of the $X(p)$ norm, the estimates on the second and third order spatial derivative derived in the  previous lemma are not strong enough for balls $B_r^d(z_0)$ that are relatively far away from the boundary in the sense that $\sqrt{\rho(z_0)}\gg r$. Estimates in such ball as well as uniform bounds on $w$ and $\grad w$ are derived in the lemma that follows.

\begin{lemma}\label{Lemma13}
\begin{enumerate}
\item Suppose that $\spt f\subset \widehat Q_r^d(z_0)$ for some $z_0\in \overline {B_1(0)}$ and $0<r\lesssim1$ and let $p>N+4$.  
Then it holds for any $0<t\lesssim r^4$ that
\[
|w(t,z_0)| + |\grad w(t,z_0)| \lesssim \|f\|_{Y(p)}.
\]
If moreover $\sqrt{\rho(z_0)} \gg r$, then it holds
\[
\begin{aligned}
\MoveEqLeft r \theta(r,z_0) |Q_r^d(z_0)|^{-\frac1p} \|\grad^2 w\|_{L^p(Q_r^d(z_0))}\\
& + r^2 \theta(r,z_0)^2 |Q_r^d(z_0)|^{-\frac1p} \|\grad^3 w\|_{L^p(Q_r^d(z_0))} \lesssim \|f\|_{Y(p)}.
\end{aligned}
\]
\item Suppose that $\spt f\subset \widehat Q(T)$ for some $T\ge 1$. Then it holds for any $p>1+N/2$ that
\[
\|w\|_{L^{\infty}(Q(T))}  + \|\grad w\|_{L^{\infty}(Q(T))}\lesssim \|f\|_{Y(p)}.
\]
\end{enumerate}
\end{lemma}

\begin{proof} 1.
As a consequence of Lemma \ref{Lemma10} and H\"older's inequality, we have that
\begin{equation}
\label{45}
|\partial_z^{\beta} w(t,z)| \le \left(\int_0^{r^4} \|\partial_z^{\beta} G(\tau,z,\cdot)\|_{L^q_{q\sigma}}^q\, d\tau\right)^{1/q} \|f\|_{L^p},
\end{equation}
where $q$ is such that $1/p+1/q = 1$ and $\beta\in \N_0^N$. The statements thus follow from suitable estimates for the kernel functions. From Proposition \ref{Prop2} we recall that
\begin{equation}
\label{38}
|\partial_z^{\beta}G(\tau,z,z')|\lesssim \sqrt[4]{\tau}^{\,-|\beta|} \theta(\sqrt[4]{\tau},z)^{-|\beta|} |B_{\sqrt[4]{\tau}}^d(z)|_{\sigma}^{-1} e^{-C\left(d(z,z')/\sqrt[4]{\tau}\right)^{4/3}}.
\end{equation}
Let $\left\{B_{j\sqrt[4]{\tau}}^d(z)\right\}_{j\in J}$ be a finite cover of $\overline B_1$. Then
\[
\int e^{-qC\left(d(z,z')/\sqrt[4]{\tau}\right)^{4/3}}\, d\mu_{q\sigma}(z') \le \sum_{j\in J} e^{-qC(j-1)^{4/3}} |B_{j\sqrt[4]{\tau}}^d(z)|_{q\sigma}.
\]
Notice that by the virtue of \eqref{8f},
\[
|B_{j\sqrt[4]{\tau}}^d(z)|_{q\sigma} \lesssim j^{2N} |B_{\sqrt[4]{\tau}}^d(z)|_{q\sigma} \sim j^{2N} |B_{\sqrt[4]{\tau}}^d(z)|^{1-q} |B_{\sqrt[4]{\tau}}^d(z)|_{\sigma}^q,
\]
which in turn implies
\[
 \int e^{-qC\left(d(z,z')/\sqrt[4]{\tau}\right)^{4/3}}\, d\mu_{q\sigma}(z') \lesssim \left(\sum_{j\in J}e^{-q(j-1)^{4/3}} j^{2N}\right) |B_{\sqrt[4]{\tau}}^d(z)|^{1-q} |B_{\sqrt[4]{\tau}}^d(z)|_{\sigma}^q .
\]
The sum is converging and can thus be absorbed in the (suppressed) constant.
We now integrate \eqref{38} over time and space and obtain
\begin{equation}
\label{39}
\int_0^{r^4} \|\partial_z^{\beta} G(\tau,z,\cdot)\|_{L^q_{q\sigma}}^q\, d\tau 
\lesssim \int_0^{r^4} \sqrt[4]{\tau}^{\, -|\beta|q} \theta(\sqrt[4]{\tau},z)^{-|\beta|q} |B_{\sqrt[4]{\tau}}^d(z)|^{1-q}\, d\tau
\end{equation}
for any $z\in \overline B_1$.

First, if $\sqrt{\rho(z)}\lesssim r $, then by \eqref{8b}, estimate \eqref{39} turns into
\begin{align*}
\int_0^{r^4} \|\partial_z^{\beta}G(\tau,z,\cdot)\|_{L_{q\sigma}^q}^q\, d\tau  
&\lesssim \int_0^{r^4} \sqrt[2]{\tau}^{-q|\beta| - (q-1)N}\, d\tau \lesssim r^{4 -2q|\beta| - 2(q-1)N},
\end{align*}
provided that $N+2 <(2-|\beta|)p$, which is consistent with the assumptions in the lemma only if $|\beta|\in \{0,1\}$. It remains to notice that
\[
r^{4 - 2(q-1)N} \lesssim r^{4q}|Q_r^d(z)|^{1-q}
\]
by the virtue of \eqref{8f}. From this and \eqref{45}, we easily derive the first estimate in the first part of the lemma in the case where $\sqrt{\rho(z)} \lesssim r$.

Second, if $\sqrt{\rho(z)}\gg r$, then \eqref{39} becomes
\begin{align*}
\int_0^{r^4} \|\partial_z^{\beta} G(\tau,z,\cdot)\|_{L^q_{q\sigma}}^q\, d\tau &\lesssim \sqrt{\rho(z)}^{\, -|\beta|q -(q-1)N} \int_0^{r^4} \sqrt[4]{\tau}^{-|\beta|q -N(q-1)}\, d\tau\\
& \lesssim \sqrt{\rho(z)}^{\, -|\beta|q -(q-1)N} r^{4 - |\beta|q -(q-1)N},
\end{align*}
provided that $N+4<(4-|\beta|)p$, which is consistent with the assumptions only if $|\beta|\in \{0,1,2,3\}$. Now we notice that
\[
\sqrt{\rho(z)}^{\, -|\beta|q -(q-1)N} r^{4 - |\beta|q -(q-1)N} \lesssim r^{(4-|\beta|)q}\theta^{-|\beta|q} |Q_r^d(z)|^{1-q},
\]
using \eqref{8f} again. It is not difficult to see that the latter estimates in combination with \eqref{45} imply remaining estimates in the first part of the lemma.

2. By Duhamel's principle in Lemma \ref{Lemma10} and the fact that $f$ is concentrated on $\widehat Q(T)$, we have for any $(t,z)\in Q(T)$ and $\beta\in \N_0^N$ that
\begin{align*}
|\partial_z^{\beta}w(t,z)| & \le \int_{\frac{T}4}^{t-1} \int |\partial_z^{\beta}G(t-t',z,z')||f(t',z')|\, d\mu_{\sigma}(z')dt'\\
& \quad +  \int_{t-1}^{t} \int |\partial_z^{\beta}G(t-t',z,z')||f(t',z')|\, d\mu_{\sigma}(z')dt',
\end{align*}
with the convention that the first integral is zero if $\frac{T}4\ge t-1$. If it is non-zero, we use   Proposition \ref{Prop2}(5) and estimate the latter   by
\[
\int_{\frac{T}4}^{t-1}\int |f|\, d\mu_{\sigma} dt' \lesssim T^{1-\frac1p} \|f\|_{L^p(\widehat Q(T))}\le T \|f\|_{L^p(\widehat Q(T))}.
\]
Similarly, applying the same strategy as in part 1 above, we bound the second term by
\[
\|\partial_z^{\beta}G(\cdot,z,\cdot)\|_{L^q((0,1);L^q_{q\sigma})} \|f\|_{L^p(\widehat Q(T))} \lesssim \|f\|_{L^p(\widehat Q(T))}.
\]
The statement thus follows by choosing $|\beta|\in \{0,1\}$.
 \end{proof}

We need some estimates for the off-diagonal parts.

\begin{lemma}\label{Lemma14}
\begin{enumerate}
\item Suppose that $\spt f\subset \left[0,r^4\right)\times \overline{B_1(0)}\setminus \widehat Q_{r}^d(z_0)$ for some $z_0\in \overline{B_1(0)}$ and $0<r\lesssim1 $. Then it holds for any   $p\in (1,\infty)$ that
\[
\|w\|_{L^{\infty}(Q_r^d(z_0))} +\|\grad w\|_{L^{\infty}(Q_r^d(z_0))} +\sum_{(\ell,k,|\beta|)\in \E} \frac{r^{4k+|\beta|}}{\theta(r,z_0)^{2\ell - |\beta|}} \|\rho^{\ell} \partial_t^k\partial_z^{\beta}  w \|_{L^p(Q_r^d(z_0))} \lesssim \|f\|_{Y(p)}.
\]
\item Suppose that $\spt f\subset \left[\frac12,\frac{T}4\right]\times \overline{B_1(0)}$ for some $T\ge 2$. Then it holds for any $p\in(1,\infty)$ that
\[
\|w\|_{L^{\infty}(Q(T))}+ \|\grad w\|_{L^{\infty}(Q(T))} + \sum_{(\ell,k,|\beta|)\in\E} T\|\rho^{\ell}\partial_t^k\partial_z^{\beta}  w\|_{L^p(Q(T))} \lesssim \|f\|_{Y(p)}.
\]
\end{enumerate}
\end{lemma}

\begin{proof}
1. We begin our proof with a helpful elementary estimates: If $\theta$ and $C$ are given positive constants, then there exists a new constant $\tilde C$ such that
\begin{equation}
\label{40}
\sqrt[4]{t-t'}^{\,-\theta} e^{-C\left(d(z,z')/\sqrt[4]
{t-t'}\right)^{4/3}} \lesssim r^{-\theta} e^{-\tilde C\left(d(z,z')/r\right)^{4/3}}
\end{equation}
for all $(t,z)\in   Q_r^d(z_0)$ and $(t',z') \in 
\left[0,r^4\right)\times \overline B_1\setminus \widehat Q_r^d(z_0)$. The argument for \eqref{40} runs as follows: To simplify the notation slightly, we write $\tau = t-t'$ and $d = d(z,z')$. If $z'\in B_{2r}^d(z_0)$, then necessarily $t'\not \in \left(r^4/4,r^4\right)$, and therefore $\tau \gtrsim r^4$. It follows that
\[
\sqrt[4]{\tau}^{\, -\theta} e^{-C\left(d/\sqrt[4]{\tau}\right)^{4/3}} \le \sqrt[4]{\tau} \lesssim r^{-\theta} e^{-C(d/r)^{4/3}},
\]
because $d(z,z') \le d(z,z_0) + d(z_0,z') \le 3r$. Otherwise, if $z'\not \in B_{2r}^d(z_0)$, it holds that $2r\le d(z',z_0)\le d(z,z') + r$, and thus $\sqrt[4]{\tau}\lesssim r \lesssim d$. Using the fact that $\tau\mapsto \sqrt[4]{\tau}^{\, -\theta} e^{-C\left(d/\sqrt[4]{\tau}\right)^{4/3}} $ is increasing for $0<\tau \lesssim d^4$, we then estimate
\[
\sqrt[4]{\tau}^{\,-\theta} e^{-C\left(d/\sqrt[4]{\tau}\right)^{4/3}} \lesssim r^{-\theta} e^{-\tilde C\left(d/r\right)^{4/3}}.
\]
This completes the proof of \eqref{40}.

In the following, $C$ will be a uniform constant whose value may change from line to line.

Because $f=0$ in $ Q_r^d(z_0)$, Duhamel's principle (Lemma \ref{Lemma10}) and the Gaussian estimates from Proposition \ref{Prop2} imply that
\[
\begin{aligned}
\MoveEqLeft |\partial_t^k\partial_z^{\beta} w(t,z)|\\
&\lesssim \int_0^t \int \frac{\sqrt[4]{\tau}^{\, -4k-|\beta|}\theta(\sqrt[4]{\tau},z)^{-|\beta|} }{|B_{\sqrt[4]{\tau}}^d(z)|_{\sigma}} e^{-C\left(d(z,z')/\sqrt[4]{\tau}\right)^{4/3}} |f(t-\tau,z')|\, d\mu_{\sigma}(z')d\tau\\
&\lesssim \int_0^{r^4} \int \left(\frac{\sqrt[4]{\tau}}{\sqrt[4]{\tau} + \sqrt{\rho(z)}}\right)^{|\beta| + N + 2\sigma }\\
&\qquad \times  \sqrt[4]{\tau}^{\, -4k-2|\beta| -2N -2\sigma}  e^{-C\left(d(z,z')/\sqrt[4]{\tau}\right)^{4/3}} |f(t-\tau,z')|\, d\mu_{\sigma}(z')d\tau.
\end{aligned}
\]
As a consequence of \eqref{40}, Remark \ref{Remark1} and the monotonicity of the function $s\mapsto \frac{s}{s+c}$ for any fixed $c>0$, we may substitute any $\sqrt[4]{\tau}$ by $r$ and find
\[
\begin{aligned}
\MoveEqLeft \frac{r^{4k+|\beta|-1}}{\theta(r,z)^{2\ell -|\beta|+1}}|\rho(z)^{\ell}\partial_t^k\partial_z^{\beta} w(t,z)|\\
& \lesssim \frac{1}{|B_r^d(z)|_{\sigma}} 
 \int_0^{r^4}\int  e^{-C\left(d(z,z')/r\right)^{4/3}}r^{-1} \theta(r,z')^{-1}|f(t',z')|\, d\mu_{\sigma}(z')dt'.
\end{aligned}
\]
We consider now a finite  family of balls $\left\{B_r^d(z_0)\right\}_{i\in I}$ covering $\overline{B_1(0)}$. Since $d(z,z_i)\le d(z,z') + r$ for any  $z'\in B_r^d(z_i)$ and
\[
\sum_{i\in I} e^{-C\left(d(z,z_i)/r\right)^{4/3}}<\infty
\]
uniformly in $r$ and $z$, we further estimate the right-hand side of the last inequality by
\begin{equation}
\label{42}
 \sup_{\tilde z\in B_1} \frac1{|B_r^d(\tilde z)|_{\sigma}} \int_0^{r^4}\int_{B^d_r(\tilde z)} r^{-1} \theta(r,z')^{-1}|f(t',z')|\, d\mu_{\sigma}(z')dt'.
\end{equation}
We claim that this term is controlled by $\|f\|_{Y(p)}$.  To see this, we fix $\tilde z\in B_1(0)$ and let $r_j = \sqrt{2/3}^j r$. Applying a non-Euclidean version of Vitali's covering lemma, cf.\ Lemma 2.2.2 in \cite{Koch99}, we find a finite family of balls $\{B_{r_j}^d(z_{ij})\}_{i\in I_j}$ covering $B_r^d(\tilde z)$ and such that
\begin{equation}
\label{41}
\sum_{i\in I_j }|B_{r_j}^d(z_{ij})|_{\sigma} \lesssim |B_r^d(\tilde z)|_{\sigma}
\end{equation}
uniformly in $j$, $r$, and $\tilde z$. Then $\left(0,r^4\right)\times B_r^d(\tilde z)$ is contained in the countable  union $ \bigcup_{j\in \N_0}\bigcup_{i\in  I_j} Q_{r_j}^d(z_{ij})$. Invoking H\"older's inequality we thus find
\[
\begin{aligned}
\MoveEqLeft \int_0^{r^4} \int_{B_r^d(\tilde z)} r^{-1}\theta(r,z')^{-1}|f(t',z')|\, d\mu_{\sigma}(z')dt'\\
& \le \sum_{j\in\N_0}\sum_{i\in I_j} \sqrt{2/3}^j r_j^{-1} \|\theta(r,\tacka)^{-1}\|_{L^q_{q\sigma}(Q_{r_j}^d(z_{ij}))} \|f\|_{L^p(Q_{r_j}^d(z_{ij}))},
\end{aligned}
\]
where, as usual, $1/p+1/q=1$. Notice that $\mu_{q\sigma}$ is a finite measure for any $q\in(1,\infty)$. From \eqref{8b} and \eqref{8c} we deduce that $\theta(r_j,z_{ij}) \sim \theta(r_j,z') \le \theta(r,z')$ for any $z'\in B_j^d(z_{ij})$, and thus, as a consequence of \eqref{8f},
\[
\|\theta(r,\tacka)^{-1}\|_{L^q_{q\sigma}(Q_{r_j}^d(z_{ij}))} \lesssim r^4_j \theta(r_j,z_{ij})^{-1} |Q_{r_j}^d(z_{ij})|^{-\frac1p} |B_{r_j}^d(z_{ij})|_{\sigma}.
\]
Combining the previous two estimates, using \eqref{41} and the convergence of the geometric series finally yields that the term in \eqref{42} bounded by $\|f\|_{Y(p)}$. We have thus proved that
\[
\frac{r^{4k+|\beta|-1}}{\theta(r,z)^{2\ell -|\beta|+1}}|\rho(z)^{\ell}\partial_t^k\partial_z^{\beta} w(t,z)|
\lesssim \|f\|_{Y(p)}.
\]
We easily deduce the statement of the lemma.

2. To prove the second statement, we use Lemma \ref{Lemma10} and Proposition \ref{Prop2}(5) to estimate
\[
|\partial_t^k\partial_z^{\beta} w(t,z)| \lesssim \int_{\frac12}^{\frac{T}4} e^{-(t-t')\lambda_1} \int |f(t',z')|\, d\mu_{\sigma}(z')dt'
\]
for any $(t,z)\in Q(T)$ and any $k\in\N_0$ and $\beta\in \N_0^N$ with $k+|\beta|\ge1$. Let $M\in \Z$ be such that $2^M\le \frac{T}4< 2^{M+1}$. We then split and compute 
\begin{align*}
|\partial_t^k\partial_z^{\beta}  w(t,z)| &\lesssim \sum_{m=0}^M \int_{2^{m-1}}^{2^m} e^{-(t-t')\lambda_1}\int |f|\, d\mu_{\sigma}dt
+ \int_{\frac{T}8}^{\frac{T}4} e^{-(t-t')\lambda_1}\int |f|\, d\mu_{\sigma}dt'\\
&\lesssim \sum_{m=0}^M e^{-(t-2^m)\lambda_1} \|f\|_{L^p(Q(2^m))} + e^{-\frac{T}4\lambda_1}\|f\|_{L^p(Q(T/4))}\\
&\lesssim e^{-\frac{T}4 \lambda_1} \|f\|_{Y(p)}.
\end{align*}
We easily infer all estimates but the uniform bound on $w$. To gain control on $\|w\|_{L^{\infty}}$, we argue similarly and get
\[
|w(t,z)| \lesssim \sum_{m=0}^{M+1} \int_{2^{m-1}}^{2^m} \int |f|\, d\mu_{\sigma}dt' \lesssim \left(\sum_{m=0}^{M+1} \frac1{(2^p)^m}\right) \|f\|_{Y(p)}.
\]
The desired estimate follows from the convergence of the geometric series.
\end{proof}

A combination of the results in this subsection yields the maximal regularity estimate for the inhomogeneous problem \eqref{32}.

\begin{prop}\label{Prop4}
Suppose that $p>N+4$. Let $w$ be a solution to the homogeneous problem \eqref{32}. Then
\[
\|w\|_{L^{\infty}(W^{1,\infty})} + \|w\|_{X(p)} \lesssim \|f\|_{Y(p)}.
\]
\end{prop}

\begin{proof}
The statement follows immediately from Lemmas \ref{Lemma11}, \ref{Lemma13} and \ref{Lemma14} and the superposition principle for linear equations: For small times, we split $f$ into $\eta f +(1-\eta)f$ with $\eta$ being a smooth cut-off function such that $\eta=1$ on $Q_r^d(z_0)$ and $\eta=0$ outside $\widehat Q_r^d(z_0)$ for some arbitrarily fixed $r\lesssim 1 $ and $z_0\in \overline{B_1(0)}$. For large time, we make a hard temporal cut-off by splitting $f$ into $\chi f + (1-\chi)f$, where $\chi $ is the characteristic function on $\widehat Q(T)$. Notice that to estimate the large times, it is enough to study such $f$'s that are zero in the initial time interval $(0,1/2)$. For details, we refer to \cite{Seis15b}.
\end{proof}

\section{The nonlinear problem}\label{S:Nonlinear}

Our goal is this section is the derivation of Theorems \ref{T1} and \ref{T2}. 
The existence of a unique solution to the nonlinear problem is a consequence of a fixed point argument. We need the following lemma:

\begin{lemma}\label{Lemma15}
Let $w_1$ and $w_2$ be two functions satisfying 
\begin{equation}
\label{44}
\|w_i\|_{L^{\infty}(W^{1,\infty})} + \|w_i\|_{X(p)} \le \eps,\quad i=1,2,
\end{equation}
for some small $\eps>0$, then
\[
\|f[w_1] - f[w_2]\|_{Y(p)} \lesssim \eps \left(\|w_1-w_2\|_{L^{\infty}(W^{1,\infty})} + \|w_1-w_2\|_{X(p)}\right).
\]
\end{lemma}

\begin{proof}
For notational convenience, we write $f^j_i = f^j[w_i]$ for any $i\in\{1,2\}$ and $j\in\{1,2,3\}$. We will also just write $w$ instead of $w_1$ or $w_2$ if the index doesn't matter. We remark that by the virtue of \eqref{44}, it holds that
\[
|R_k[w_1]-R_k[w_2]|\lesssim \|w_1-w_2\|_{L^{\infty}(W^{1,\infty})}
\]
and 
\[
|R_k[w]|\lesssim 1
\]
of any value of $k$.

The estimates of the differences of the $f^j_i$ is very similar. We focus on the leading order terms, i.e., $f^3_1-f_2^3$. Using \eqref{44} and the previous bounds on the $R_k$'s, we first notice that
\begin{align*}
|f_1^3-f_2^3| &\lesssim \rho^2 \|w_1-w_2\|_{L^{\infty}(W^{1,\infty})} \left(|\grad^2 w|^3 + |\grad^2 w||\grad^3 w| + |\grad w||\grad^4 w|\right)\\
&\quad +\rho^2  |\grad^2 w_1-\grad^2 w_2| \left(|\grad^2 w|^2 + |\grad w||\grad^3 w|\right)\\
&\quad+ \rho^2 |\grad^3 w_1-\grad^3 w_2| |\grad w||\grad^2 w| + \rho^2 |\grad^4 w_1 - \grad^4 w_2||\grad w|^2.
\end{align*}
The control of the individual terms is derived very similarly. There are a few obvious cases, for instance the last term, which is simply controlled by using \eqref{44}:
\[
\|\rho^2 |\grad^4 w_1 - \grad^4 w_2||\grad w|^2\|_{Y(p)} \le \|w\|_{L^{\infty}(W^{1,\infty})}^2 \|w_1-w_2\|_{X(p)} \le \eps \|w_1-w_2\|_{X(p)}.
\]
For most of the remaining terms, we have to make use of the following interpolation inequality
\[
\|\grad^i \zeta\|_{L^r_{\sigma}}^m \lesssim \|\zeta\|_{L^{\infty}}^{m-i} \|\grad^m\zeta\|_{L^p_{\sigma}}^i,
\]
provided that $ mp = ir$ for some integers $i<m$,
which has been proved in Appendix A of \cite{Seis15b}. For instance, setting $\zeta =\eta \grad w$ for some smooth cut-off function $\eta$ satisfying $\eta=1$ in $B_r^d(z_0)$ and $\eta = 0$ outside $B_{2r}^d(z_0)$, we have that
\[
\|\rho^2 |\grad^2 w|^3\|_{L^p(B_r^d(z_0))}  = \|\grad^2 w\|_{L^{3p}_{2p}(B_r^d(z_0))}^3 \le \|\grad \zeta\|_{L_{2p}^{3p}}^3.
\]
Applying the above interpolation inequality and using the fact that $\eta$ varies on the scale $r \theta(r,z_0)$ and $\rho \lesssim \theta(r,z_0)^2$ in $B_{2r}^d(z_0)$ (see \eqref{8b} and \eqref{8c}), we then get
\begin{align*}
\|\rho^2 |\grad^2 w|^3\|_{L^p(B_r^d(z_0))} &\lesssim \|\grad w\|_{L^{\infty}}^2\left(\|\rho^2 \grad^4 \|_{L^p(B_{2r}^d(z_0))} + \frac{\theta}{r} \|\rho \grad^3 w\|_{L^p(B_{2r}^d(z_0))}\right.\\
&\quad \quad \left. + \ \frac{\theta^2}{r^2}\| \grad^2 w\|_{L^p(B_{2r}^d(z_0))} + \frac{\theta}{r^3} \| \grad w\|_{L^p(B_{2r}^d(z_0))}\right),
\end{align*}
where $\theta = \theta(r,z_0)$. Integrating in time over $\left(\frac{r^4}2,r^4\right)$, multiplying by $r^3/\theta$ and using \eqref{44} then yields
\[
\sup_{r,z_0} \frac{r^3}{\theta} |Q_r^d(z_0)|^{-\frac1p} \|\rho^2 |\grad^2 w|^3\|_{L^p(Q _r^d(z_0))} \lesssim \eps \|w\|_{X(p)}.
\]
This type of estimate can be used, for instance, to bound the first term in the above estimate for $f_1^3-f_2^3$ for small times. The remaining terms and the large time parts of the $Y(p)$ norm can be controlled in a similar way. 
\end{proof}

We are now in the position to prove Theorems \ref{T1} and  \ref{T2}.

\begin{proof}[Proof of Theorem \ref{T1} and \ref{T2}]
 To simplify the notation in the following, we denote by $\bar X(p)$ the intersection $X(p)\cap L^{\infty}(W^{1,\infty})$ and set $\|\cdot\|_{\bar X(p)} = \|\cdot \|_{X(p)} + \|\cdot \|_{L^{\infty}(W^{1,\infty})}$. 
Let $\eps$ and $\eps_0$ be two positive constants. We denote by $M_{\eps}$ the set of all functions $w$ in $\bar X(p)$ such that $\|w\|_{\bar X(p)}\le \eps$ and by $N_{\eps_0}$ the set of all functions $g$ such that $\|g\|_{W^{1,\infty}}\le \eps_0$. We divide the proof into several steps.

\emph{Step 1. Existence and uniqueness.}
For $w\in M_{\eps}$ and $g\in N_{\eps_0}$ given, we denote by $\tilde w := I(w,g)$ the unique solution to the linear problem \eqref{32} with inhomogeneity $f = f[w]$. By Theorem \ref{Thm1}, we have the estimate $\|\tilde w\|_{\bar X(p)}\lesssim \|f[w]\|_{Y(p)} + \|g\|_{W^{1,\infty}}$. Applying Lemma \ref{Lemma15} with $w_1=w$ and $w_2=0$ and using the assumptions on $w$ and $g$, we find that $\|\tilde w\|_{\bar X(p)} \le C(\eps^2 +\eps_0)$ for some positive constant $C$. We choose $\eps$ and $\eps_0$ small enough so that $C\eps^2 \le \frac{\eps}2$ and $C\eps_0\le \frac{\eps}2$, with the consequence that $\tilde w\in M_{\eps}$. This reasoning implies that for any fixed $g\in N_{\eps_0}$, the function $\tilde w(\cdot,g)$ maps the set $M_{\eps}$ into itself. Moreover, given $w_1$ and $w_2$ in $M_{\eps}$, we find by linearity and Lemma \ref{Lemma15} that 
\[
\|I(w_1,g) - I(w_2,g)\|_{\bar X(p)} \lesssim \|f[w_1]-f[w_2]\|_{Y(p)} \lesssim \eps \|w_1-w_2\|_{\bar X(p)}.
\]
Thus, choosing $\eps$ even smaller, if necessary, the previous estimate shows that $I(\cdot ,g)$ is a contraction on $M_{\eps}$. By Banach's fixed point argument, there exists thus a unique $w^*\in M_{\eps}$ such that $w^* =I(w^*,g)$. In particular, $w^*$ solves the nonlinear equation. From the previous choice of $\eps$, we moreover deduce that $
\|w^*\|_{\bar X(p)} \lesssim \|g\|_{W^{1,\infty}}$.

\emph{Step 2. Analytic dependence on initial data.} In order to show that $w^*$ depends analytically on $g$, we will apply the analytic implicit function theorem, cf.~\cite[Theorem 15.3]{Deimling85}. Because the nonlinearity $f=f[w]$ is a rational function of $w$ and $\grad w$, and thus analytic away from its poles, the contraction map $I$ is an analytic function on $M_{\eps}\times N_{\eps_0}$. We consider the map $J : M_{\eps}\times N_{\eps_0}\to M_{\eps}$ defined by $J(w,g) = w - I(w,g)$. Because $I$ is analytic, so is $J$. It holds that $I(0,0)=0$ and $D_w I(0,0)  = \id$. From the analytic implicit function theorem we deduce the existence of two constants $\tilde \eps <\eps$ and $\tilde \eps_0<\eps_0$ and of an analytic map $A :N_{\tilde \eps_0}\to M_{\tilde \eps}$ with $A(0)=0$ and such that $J(w,g) = 0$ if and only if $A(g)=w$. From the uniqueness of the fixed point and the definition of $J$  we then conclude that the map $g\mapsto w^*$ is analytic from $N_{\tilde \eps_0}$ to $M_{\tilde \eps}$. 

\emph{Step 3. Analytic dependence on time and tangential coordinates.} Let us now change from Euclidean to spherical coordinates. For $z = (z_1,\dots, z_N)^T\in \overline B_1(0)$, we find radius $s\in [0,1]$ and an angle vector $\phi = (\phi_1,\dots,\phi_{N-1})^T\in \A_{N-1} := [0,\pi]^{N-2}\times [0,2\pi]$ such that $z_n = s (\prod_{i=1}^{n-1}\sin \phi_i) \cos \phi_n$ for $n\le N-1$ and $z_N = s\prod_{i=1}^{N-1} \sin \phi_i$. By a slight abuse of notation, we write $w(t,z) = w(t,s,\phi)$. For $\lambda\in \R$ and $\psi \in \A_{N-1}$ we define
\[
w_{\lambda,\psi}^*  := w^*\circ \Xi_{\lambda,\psi},\quad\Xi_{\lambda,\psi}(t,s,\phi): = (\lambda t, s, \phi+t\psi).
\]
A short computation reveals that $w_{\lambda,\psi}$ solves the equation
\[
\partial_t w_{\lambda,\psi}^* + \H_{\sigma}w_{\lambda,\psi}^* = f_{\lambda,\psi}[w_{\lambda,\psi}^*],
\]
where
\[
f_{\lambda,\psi}[w]: = \lambda f[w] + (1-\lambda) \H_{\sigma} w + \psi\cdot \grad_{\phi}w, \quad
\H_{\sigma} = \L_{\sigma}^2 + n\L_{\sigma}.
\]
Clearly, $f_{1,0} = f$. Similarly as above, we denote by $I_{\lambda,\psi}(w,g)$ the solution to the linear equation with  inhomogeneity $f_{\lambda,\psi}[w]$ and initial datum $g$. We furthermore set $J_{\lambda,\psi}(w,g)  : = w-I_{\lambda,\psi}(w,g)$. It is obvious that $J_{1,0}(0,0) = 0$ and $D_w J_{1,0}(0,0) = \id$. Applying the analytic implicit function theorem once more, we find positive constants $\delta$, $\hat \eps < \tilde \eps$, $\hat \eps_0<\tilde \eps_0$ and an analytic function $A_{\lambda,\psi}(g) = A(\lambda,\psi,g)$ from $B_{\delta}^{\R}(1)\times B_{\delta}^{\R^{N-1}}(0) \times N_{\hat \eps_0}$ to $M_{\hat \eps}$ such that $J_{\lambda,\psi}(A_{\lambda,\psi}(g),g)  =0$. In particular, the above uniqueness result entails that $A_{\lambda,\psi}(g) = A(g)\circ \Xi_{\lambda,\psi}$. We conclude that $w_{\lambda,\psi}\in \bar X(p)$ depends analytically on $\lambda$ and $\psi$ in a neighborhood of $(1,0)\in \R\times \R^{N-1}$. In particular, there exists a constant $\Lambda$ dependent only on $N$ such that  for any $k\in \N_0$ and $\beta'\in \N_0^{N-1}$, it holds that
\[
\|\partial_{\lambda}^k\partial_{\psi}^{\beta'} \big|_{(\lambda,\psi)=(1,0)} w_{\lambda,\psi} \|_{\bar X(p)} \lesssim \Lambda^{-k-|\beta'|} k!\beta! \|g\|_{W^{1,\infty}}.
\]
It remains to notice that 
\[
\partial_{\lambda}^k\partial_{\psi}^{\beta'} \big|_{(\lambda,\psi)=(1,0)} w_{\lambda,\psi} (t,z) = t^{k+|\beta'|} \partial_t^k\partial_{\phi}^{\beta'} w(t,r,\phi)
\]
to deduce
\begin{equation}
\label{47}
t^{k+|\beta'|} |\partial_t^k\partial_{\phi}^{\beta'} \grad w(t,r,\phi)| \lesssim \Lambda^{-k-|\beta'|}k!\beta'! \|g\|_{W^{1,\infty}}.
\end{equation}

\emph{Step 4. Regularity in transversal direction.}
The derivation of the transversal regularity relies on  the analyticity bounds established above together with the Morrey estimate
\begin{equation}
\label{47c}
\begin{aligned}
\|v\|_{L^{\infty}(Q_r^d(z))} & \lesssim |Q_r^d(z)|_{\sigma}^{-\frac1p} \|v\|_{L^p_{\sigma}(Q_r^d(z))}\\
&\quad + r\theta |Q_r^d(z)|_{\sigma}^{-\frac1p} \|\grad v\|_{L^p_{\sigma}(Q_r^d(z))} + r^4 |Q_r^d(z)|_{\sigma}^{-\frac1p} \|\partial_t v\|_{L^p_{\sigma}(Q_r^d(z))},
 \end{aligned}
 \end{equation}
which holds for any $p>N$ uniformly in $r$ and $z$.
The proof of this estimate proceeds analogously to the Euclidean case, see, e.g., \cite[Chapter 4.5]{EvansGariepy92}. We omit the argument.

In the following discussion, we keep $r$ and $z$ fixed and we set $\theta = \theta(r,z)$. For $b\in \{2,3\}$, we choose $\sigma = (b-1)p$ and apply \eqref{47c} to the effect that
\begin{align*}
\|\grad_{\psi}^{4-b} \partial_s^b w\|_{L^{\infty}(Q_r^d(z))} & \lesssim |Q_r^d(z)|^{-\frac1p}_{(b-1)p} \|\rho^{b-1}\grad_{\psi}^{4-b} \partial_s^b w\|_{L^p(Q_r^d(z))}\\
&\quad + r\theta |Q_r^d(z)|^{-\frac1p}_{(b-1)p} \|\rho^{b-1} \grad_{\psi}^{4-b} \partial_r^b\grad w\|_{L^p(Q_r^d(z))}\\
&\quad + r^4 |Q_r^d(z)|^{-\frac1p}_{(b-1)p} \|\rho^{b-1} \grad_{\psi}^{4-b} \partial_r^b\partial_t w\|_{L^p(Q_r^d(z))}.
\end{align*}
We recall from \eqref{8f} that $|Q_r^d(z)|_{\sigma}\sim \theta^{2\sigma} |Q_r^d(z)|$ and that $\sqrt{\rho(\tilde z)} \lesssim \theta$ for any $\tilde z\in B_r^d(z)$ by the virtue of \eqref{8c}. Therefore,
\begin{align*}
\|\grad_{\psi}^{4-b} \partial_s^b w\|_{L^{\infty}(Q_r^d(z))} & \lesssim \theta^{4-2b} |Q_r^d(z)|^{-\frac1p}  \|\rho^{b-2}\grad_{\psi}^{4-b} \partial_s^b w\|_{L^p(Q_r^d(z))}\\
&\quad + r\theta^{3-2b} |Q_r^d(z)|^{-\frac1p} \|\rho^{b-1} \grad_{\psi}^{4-b} \partial_r^b\grad w\|_{L^p(Q_r^d(z))}\\
&\quad + r^4 \theta^{4-2b} |Q_r^d(z)|^{-\frac1p} \|\rho^{b-2} \grad_{\psi}^{4-b} \partial_r^b\partial_t w\|_{L^p(Q_r^d(z))}.
\end{align*}
With the help of the analyticity estimates \eqref{47}, we easily deduce that
\begin{equation}
\label{47a}
r\theta \|t^2 \grad_{\psi}^2 \partial_s^2 w\|_{L^{\infty}(Q_r^d(z))} + (r\theta)^2 \|t \grad_{\psi}\partial_s^3 w\|_{L^{\infty}(Q_r^d(z))} \lesssim \| g\|_{W^{1,\infty}}.
\end{equation}
An analogous argument yields the corresponding control of the time derivatives, namely
\begin{equation}
\label{47b}
\frac{r^3}{\theta} \|\partial_t w\|_{L^{\infty}(Q_r^d(z))} \lesssim \| g\|_{W^{1,\infty}}.
\end{equation}
In order to deduce control over the fourth order vertical derivatives, we rewrite the nonlinear equation \eqref{50} as
\[
\rho\partial_r^2 (\rho \partial_r^2 w) = f[w] - \partial_t w + \mbox{l.o.t.}.
\]
The terms on the right-hand side are all uniformly controlled thanks to \eqref{47},\eqref{47a} and \eqref{47b}. Similarly, we may write
\[
-\rho^{-1} \partial_r(\rho^2 \partial_r v) = h
\]
for some $h$ such that $t^{\kappa} h\in L^{\infty}$ for some $\kappa>0$, and where $v = -\rho^{-1}\partial_r(\rho^2 \partial_r w)$. This identity can be integrated so that
\[
\partial_r v = \rho^{-2}\int_r^1 \rho h\, d\tilde r.
\]
The expression on the right is differentiable with
\[
\partial_r^2 v = 2\rho^{-3}r\int_r^1\rho h\, d\tilde r -\rho^{-1}h.
\]
We deduce that $\rho t^{\kappa}\partial_r^2 v\in L^{\infty}$ and thus $\rho^2 t^{\kappa} \partial_r^4 w\in L^{\infty}$. 

This argument can be iterated an yields smoothness of $w$.

\end{proof}

\section*{Appendix: Derivation of the transformed equation}
Let us write $z = \Phi_t(x)$. We will first verify that $\Phi$ defines a diffeomorphism. For this purpose, we compute the 
 derivatives of $\Phi$ in terms of $x$ and $v$,
\[
\partial_i\Phi^j = \frac{\delta_{ij}}{(2 v +|x|^2)^{1/2}} - \frac{x_j(\partial_i v + x_i)}{(2v +|x|^2)^{3/2}}.
\]
Recalling the elementary formula $\det(I-a\otimes b) = I- a\cdot b$ for any two vectors $a$ and $b$, we compute that
\[
\det \grad \phi(x) = \frac{2v  -x\cdot \grad v}{(2v+|x|^2)^{\frac{N}2+1}}.
\]
If $v$ is close to the Smyth--Hill solution in the sense that
\[
\|v-v_*\|_{L^{\infty}(\P(v))} + \|\grad v+x\|_{L^{\infty}(\P(v))} \le \epsilon
\]
for some small $\eps$, we find that $2v-x\cdot \grad v \ge 1-3\eps$ and $2v+|x|^2\ge 1-2\eps$, which implies that the Jacobi determinant is finite if $\eps$ is sufficiently small.

Let us express the derivative of $\Phi$ in terms of the new variables $z$ and $w$. Differentiating \eqref{2a} yields
\[
\partial_i v +x_i =(1+ w)\grad  w \cdot \partial_i \Phi = \partial_i  w -\frac{z\cdot \grad  w}{1+ w} (\partial_i v+x_i),
\]
and thus
\[
 \partial_i  v +x_i  = \frac{1+ w}{1+ w+z\cdot \grad  w} \partial_i  w.
\]
Plugging this and \eqref{2a} into the expression for the derivatives of $\Phi$, we find
 \[
\partial_i\Phi^j  = \frac{\delta_{ij}}{1+ w} - \frac{z_j\partial_i  w}{(1+w)(1+ w+z\cdot \grad  w)}.
\]
Under the assumption that $w$ is such that 
\[
\|w\|_{L^{\infty}} + \|\grad w\|_{L^{\infty}} \le \eps
\]
for some small $\eps$, we see by a calculation similar as the one above that $\Phi$ is a diffeomorphism.

We will now compute how the change of variables acts on the confined thin film equation \eqref{3}. For notational convenience, we set
\[
\rho(z)  =\frac12(1-|z|^2),
\]
and $\tilde w  = 1+w$, with the effect that
\begin{equation}
\label{5}
\rho\tilde w^2 = v = \gamma u^{1/2}.
\end{equation}

For an arbitrary function $f=f(z)$, it thus holds that
\begin{equation}
\label{5a}
\partial_i\left(f(\Phi)\right) = \frac{\partial_i f}{\tilde w}  -\frac{(z\cdot \grad f)\partial_i\tilde w}{\tilde w(\tilde w+z\cdot\grad\tilde  w)}.
\end{equation}
Now, differentiating \eqref{5} with respect to $x_i$ yields
\begin{eqnarray*}
\gamma^2\partial_i u  &=& \frac1{\tilde w} \partial_i (\rho^2\ \tilde w^4) - \frac{\partial_i\tilde w}{\tilde w(\tilde w + z\cdot \grad \tilde w)} z\cdot \grad(\rho^2 \tilde w^4)\\
&=& -2\rho\tilde w^3 z_i + 2\frac{\rho \tilde w^3 \partial_i \tilde w}{\tilde w + z\cdot \grad \tilde w}.
\end{eqnarray*}
Differentiating with respect to $x_i$ again, we obtain that
\begin{eqnarray*}
\frac{\gamma^2}2\partial_i^2 u &=& -(\rho-z_i^2)\tilde w^2 + \frac{\tilde w^2}{\tilde w+z\cdot \grad\tilde w} \rho^{-1} \partial_i(\rho^2\partial_i \tilde w)\\
&&\mbox{} -\frac{\tilde w^2}{\tilde w +z\cdot \grad\tilde w} \rho z\cdot\grad\left(\frac{(\partial_i \tilde w)^2}{\tilde w+z\cdot \grad\tilde w}\right)\\
&&\mbox{} + (\rho + |z|^2) \frac{\tilde w^2(\partial_i \tilde w)^2}{(\tilde w +z\cdot \grad\tilde w)^2}.
\end{eqnarray*}
Hence, summing over $i$ and rearranging terms yields
\begin{align*}
\frac{\gamma^2}2\laplace u \times \frac{\tilde w + z\cdot \grad\tilde w}{\tilde w^2}&= (1-(N+2)\rho)(\tilde w + z\cdot \grad\tilde w)  -\L \tilde w \\
&\quad + (1-\rho) \frac{|\grad\tilde w|^2}{\tilde w + z\cdot \grad\tilde w} - \rho z\cdot \grad \left(\frac{|\grad\tilde w|^2}{\tilde w+z\cdot \grad\tilde w}\right).
\end{align*}
 With the help of the $\star$-notation, the (nonlinear) term in the second line of the above identity can be rewritten as
\[
 p\star \sum_{k=1}^2 \frac{(\grad\tilde w)^{(k-1)\star}}{(\tilde w +z\cdot \grad\tilde w)^k} \left((\grad\tilde w)^{2\star} + \rho\grad\tilde w\star \grad^2\tilde w\right).
\]
In what follows, it should become clear why this way or writing drastically simplifies the notation.

With the help of \eqref{5a}, we compute 
\[
\partial_i \left(\left(\frac{\tilde w^2}{\tilde w+z\cdot \grad\tilde w}f\right)(\Phi)\right) = \frac{\tilde w}{\tilde w+z\cdot \grad\tilde w}\left(\partial_i f - z\cdot \grad\left(\frac{\partial_i\tilde w f}{\tilde w+z\cdot\grad\tilde w}\right)\right)
\]
for any function $f=f(z)$, and thus
\[
\begin{aligned}
\MoveEqLeft \frac{\gamma^2}2\left(\partial_i \laplace u-x_i\right) \times \frac{\tilde w + z\cdot \grad \tilde w}{ \tilde w}\\
 & =  - N\partial_i\tilde w -\partial_i\L\tilde w\\
&\quad + p\star \frac{ (\grad \tilde w)^{2\star} +  \grad \tilde w\star \grad^2\tilde w + \rho  (\grad^2\tilde w)^{2\star} + \rho \grad\tilde w\star \grad^3\tilde w}{\tilde w+z\cdot \grad\tilde w}\\
&\quad + p\star \frac{(\grad \tilde w)^{3\star} + (\grad \tilde w)^{2\star}\star\grad^2\tilde w + \rho \grad\tilde w\star(\grad^2\tilde w)^{2\star} +\rho (\grad\tilde w)^{2\star}\star\grad^3\tilde w}{(\tilde w+z\cdot \grad\tilde w)^2}\\
&\quad + p\star\frac{(\grad\tilde w)^{4\star} + (\grad\tilde w)^{3\star}\star \grad^2\tilde w + \rho (\grad \tilde w)^{2\star}\star(\grad^2\tilde w)^{2\star} +\rho (\grad\tilde w)^{3\star}\star\grad^3\tilde w}{(\tilde w+z\cdot \grad\tilde w)^3}\\
&\quad + p\star \frac{(\grad\tilde w)^{5\star} + \rho(\grad\tilde w)^{4\star} \star\grad^2\tilde w + \rho(\grad\tilde w)^{3\star}\star(\grad^2\tilde w)^{2\star}}{(\tilde w+z\cdot \grad\tilde w)^4}.
\end{aligned}
\]
We notice that the nonlinearity belongs to the class
\[
p\star \sum_{k=1}^4 \frac{(\grad \tilde w)^{(k-1)\star}}{(\tilde w +z\cdot \grad\tilde w)^k} \star \left((\grad\tilde w)^{2\star} + \grad\tilde w\star \grad^2\tilde w + \rho(\grad^2\tilde w)^{2\star} +\rho\grad\tilde w\star\grad^3\tilde w\right).
\]

Similarly as above, we compute for an arbitrary function $f=f(z)$ that
\[
\begin{aligned}
\MoveEqLeft \partial_i \left(\left(\frac{\tilde w^5}{\tilde w+z\cdot \grad\tilde w}f\right)(\Phi)\right)\\ &= \frac{\tilde w^4}{\tilde w+z\cdot \grad\tilde w}\left(\partial_i f  + 3 \frac{\partial_i\tilde w f}{\tilde w+z\cdot \grad\tilde w} -z\cdot \grad\left( \frac{\partial_i\tilde w f}{\tilde w+z\cdot\grad\tilde w}\right)\right),
\end{aligned}
\]
and thus
\begin{align*}
\MoveEqLeft \frac{\gamma^4}2\div\left(u\grad\laplace u - ux\right) \times \frac{\tilde w+z\cdot\grad \tilde w}{\rho\tilde w^4} \\
& = (N+\L)\L\tilde w\\
&\quad  + p\star\tilde R_{-1}[\tilde w] \star \left((\grad \tilde w)^{2\star} +  \grad\tilde w\star \grad^2\tilde w \right)\\
&\quad  + p\star \tilde R_{-1}[\tilde w] \star  \rho\left( (\grad^2\tilde w)^{2\star} + \grad\tilde w\star \grad^3\tilde w \right)\\
&\quad + p\star \tilde R_{-2}[\tilde w] \star  \rho^2\left( (\grad^2\tilde w)^{3\star}+ \grad\tilde w\star  \grad^2\tilde w \star \grad^3 \tilde w + (\grad\tilde w)^{2\star} \star\grad^4\tilde w\right),
\end{align*}
where $\tilde R_i[\tilde w] = r_i(\grad \tilde w,\tilde w+z\cdot\grad\tilde w)$ for some rational functions $r_i$ that are homogeneous of degree $i$, i.e., $r_i(s a,s b) = s^ir_i(a,b)$.

We finally turn to the computation of the time derivative. For this notice first that
\[
\partial_t \Phi_t(x) = -\frac{\gamma^2}2 \frac{z}{\rho \tilde w^4}\partial_t u,
\]
and thus, a short computation shows that
\[
\frac{\gamma^2}2\partial_t u = \frac{\rho\tilde w^4}{\tilde w+z\cdot \grad \tilde w}{\partial_t \tilde w}.
\]

After a rescaling of time $t\to \gamma^2  t$, and recalling that $\tilde w = 1+w$, we find the transformed equation \eqref{50}.

\section*{Acknowledgement}
The author thanks Herbert Koch for helpful discussions.

\bibliography{tfe_lit}
\bibliographystyle{abbrv}
\end{document}